\documentclass{article}

\newif\ifarxiv
\arxivtrue

\usepackage{graphicx}
\usepackage{amsmath,amssymb,amsthm,bm}
\usepackage[a4paper]{geometry}
\usepackage{algorithm,algorithmic}
\usepackage[colorlinks=false,pdfborder={0 0 0}]{hyperref}
\usepackage{authblk}
\usepackage{color}

\graphicspath{{fig/}}
\DeclareGraphicsExtensions{.mps,.pdf,.eps,.png}

\newenvironment{keywords}{\medskip\textbf{Keywords:}}{}
\newenvironment{AMS}{\medskip\textbf{AMS subject classifications (2020).}}{}

\newtheorem{theorem}{Theorem}

\newtheorem{remark}{Remark}
\newtheorem{lemma}{Lemma}
\newtheorem{proposition}{Proposition}

\theoremstyle{plain}

\DeclareMathOperator{\diag}{diag}

\newcommand{\fro}{\mathsf F}

\newcommand*{\set}[1]{\left\lbrace#1\right\rbrace}

\newcommand*{\herm}{^*}

\newcommand*{\iherm}{^{-*}}

\newcommand{\bmat}[1]{\begin{bmatrix}#1\end{bmatrix}}

\newcommand*{\conj}[1]{\bar{#1}}

\def\adots{\mathinner{\mkern2mu\raise1pt\hbox{.}\mkern2mu
    \raise4pt\hbox{.}\mkern2mu\raise7pt\hbox{.}\mkern1mu}}

\newcommand*{\macheps}{\bm u}
\newcommand*{\machepslow}{\bm{u}_{\mathrm{low}}}
\newcommand*{\epsqr}{\varepsilon_{\mathrm{qr}}}
\newcommand*{\epsrrqr}{\varepsilon_{\mathrm{rrqr}}}

\newcommand*{\epsJh}{\varepsilon_{\mathrm J}}
\newcommand*{\epsJbar}{\bar\varepsilon_{\mathrm J}}
\newcommand*{\epsSVD}{\varepsilon_{\mathrm{SVD}}}
\newcommand*{\epsU}{\varepsilon_{\mathrm U}}
\newcommand*{\epsQ}{\varepsilon_{\mathrm Q}}

\newcommand*{\tol}{\mathtt{tol}}
\newcommand*{\tolalg}{\mathtt{tol}_{\mathrm{alg}}}
\newcommand*{\tolcond}{\mathtt{tol}_{\mathrm{cond}}}
\newcommand*{\tolorth}{\mathtt{tol}_{\mathrm{orth}}}

\newcommand*{\single}{\mathtt{lower}}
\newcommand*{\double}{\mathtt{working}}

\newcommand*{\s}{\mathrm{low}}
\newcommand*{\bigO}{O}
\DeclareMathOperator{\fl}{f{}l}
\DeclareMathOperator{\up}{up}
\DeclareMathOperator{\low}{low}
\DeclareMathOperator{\off}{off}
\newcommand*{\fllow}{\fl_{\mathrm{low}}}

\begin{document}

\ifarxiv
\title{A mixed precision Jacobi SVD algorithm%
}
\author[1,2,3]{Weiguo Gao}
\author[4]{Yuxin Ma}
\author[2,3]{Meiyue Shao}
\affil[1]{School of Mathematical Sciences, Fudan University, Shanghai 200433,
China}
\affil[2]{School of Data Science, Fudan University, Shanghai 200433, China}
\affil[3]{MOE Key Laboratory for Computational Physical Sciences, Fudan
University, Shanghai 200433, China}
\affil[4]{Department of Numerical Mathematics, Faculty of Mathematics and Physics, Charles University, Sokolovsk\'{a} 49/83, 186 75 Praha 8, Czechia}
\maketitle
\fi

\begin{abstract}
We propose a mixed precision Jacobi algorithm for computing the singular value
decomposition (SVD) of a dense matrix.
After appropriate preconditioning, the proposed algorithm computes the SVD in
a lower precision as an initial guess, and then performs one-sided Jacobi
rotations in the working precision as iterative refinement.
By carefully transforming a lower precision solution to a higher precision one,
our algorithm achieves about \(2\times\) speedup on the x86-64 architecture
compared to the usual one-sided Jacobi SVD algorithm in LAPACK, without
sacrificing the accuracy.

\ifarxiv

\begin{keywords}
Singular value decomposition,
Jacobi algorithm,
mixed precision algorithm,
high relative accuracy
\end{keywords}

\begin{AMS}
65F15
\end{AMS}
\else

\fi
\end{abstract}

\section{Introduction}
\label{sec:introduction}
In recent years, partly because of the quick development of machine learning,
low precision arithmetic becomes more and more widely supported by hardware on
various computational units such as GPUs.
Since lower precision arithmetic is in general faster than higher precision
arithmetic, the availability of low precision arithmetic motivates researchers
to rethink about efficiency and accuracy of existing numerical algorithms.
It is possible to make significant speedup if an algorithm can make clever use
of low precision arithmetic.
As a consequence, mixed precision algorithms have attracted a lot of attention,
especially in numerical linear algebra and high performance
computing~\cite{Survey2021,AADGHKLTYY2016,HM2022}.

There are mainly two classes of mixed precision algorithms.
Traditional mixed precision algorithms, such as classical iterative refinement
for linear systems~\cite{JW1977,Moler1967} and Cholesky QR
algorithm~\cite{YTD2015}, improve the numerical stability of the algorithm by
temporarily using a higher precision at certain points.
In contrast with traditional ones, modern mixed precision algorithms use a
lower precision to perform part of the computations in order to reduce the
execution time of programs, while preserving accuracy of the final solution
the same as that of fixed precision algorithms.

There are a number of works in the literature focusing on mixed precision
linear solvers.
In~2018, a general framework for solving a nonsingular linear system \(Ax=b\)
based on iterative refinement with three precisions was proposed by Carson and
Higham~\cite{CH2018}.
By computing the most expensive operation, LU factorization of \(A\), in a
lower precision, a mixed precision algorithm can solve linear systems up to
twice as fast as traditional dense linear solvers without loss of accuracy.
Refinement through GMRES (GMRES-IR) utilizing lower precision LU factorization
as the preconditioner is also developed for ill-conditioned linear system, in
order to obtain higher performance than traditional dense linear solvers.
More recently, Carson and Khan~\cite{CK2022} proposed that mixed precision
GMRES-IR can be further accelerated by sparse approximate inverse
preconditioners.
Carson, Higham, and Pranesh also proposed three-precision GMRES-IR with QR
preconditioning to solve linear least squares problems through augmented
systems~\cite{CHP2020}.

Compared to dense linear solvers, iterative refinement for dense eigensolvers
(that computes a large number of eigenpairs) is more complicated, though most
eigensolvers are essentially iterative.
Since several popular algorithms (such as the QR algorithm) are highly
structured, and cannot make full use of a near diagonal or near upper
triangular input, it seems difficult to use these algorithms directly for
iterative refinement.
Ogita and Aishima recently published several papers on Newton-like methods for
refining the solution of symmetric eigenvalue problems, including singular
value decomposition~\cite{OA2018,OA2019,OA2020}.
A similar but more complicated algorithm was proposed by Bujanovi\'c, Kressner,
and Schr\"oder to refine the Schur decomposition for nonsymmetric eigenvalue
problems~\cite{BKS2022}.
This refinement algorithm also leads to a mixed precision nonsymmetric
eigensolver.
There is another recent algorithm called iterative perturbative theory (IPT)
developed by Kenmoe, Kriemann, Smerlak, and Zadorin in~\cite{KKSZ2021} that
works for both symmetric and nonsymmetric eigenvalue problems.

Besides Newton-like algorithms, the Jacobi algorithm~\cite{Jacobi1846} is also
suitable for iterative refinement~\cite{BO2020,Wilkinson1968} since it
converges locally quadratically~\cite{Henrici1958}.
A natural idea is to quickly compute the solution in lower precision, and then
use the Jacobi algorithm to improve the accuracy to working precision.
This leads to a basic mixed precision eigensolver.
However, it is challenging to achieve high performance in this manner, as both 
preprocessing and refinement stages require \(O(n^3)\) work.
We observe that a basic mixed precision eigensolver usually has
relatively low performance (see Section~\ref{subsec:experiments-simplemsvj}).
In order to accelerate the computation---this is the motivation of exploiting
low precision arithmetic---it is crucial to carefully design the
algorithm so that the Jacobi algorithm can quickly refine the
low-precision solution.

In this work we propose a mixed precision Jacobi SVD algorithm.
Our algorithm makes use of low precision arithmetic as a preconditioning step,
and then refines the solution by the one-sided Jacobi algorithm developed by
Drma\v{c} and Veseli\'{c} in~\cite{DV2008a,DV2008b}.
On the x86-64 architecture our mixed precision algorithm is in general about
twice as fast as the fixed precision one in LAPACK.
Moreover, the mixed precision algorithm inherits high accuracy properties
of the Jacobi algorithm even if a large amount of work is performed in a
lower precision.
As an eigensolver, the Jacobi algorithm is usually not the first choice since
it is often slower than alternatives such as the divide and conquer
algorithm~\cite{BO2020, KKSZ2021}.
However, when high relative accuracy is desired, the Jacobi algorithm seems the
best candidate by far~\cite{Demmel1997, DV1992}.

The rest of this paper is organized as follows.
In Section~\ref{sec:preliminary} we briefly review the one-sided Jacobi SVD
algorithm.
Then in Section~\ref{sec:algorithm} we propose a mixed precision Jacobi SVD
algorithm and discuss algorithmic details.
Theoretical analyses on the mixed precision Jacobi SVD algorithm are provided
in Sections~\ref{sec:accuracy}--\ref{sec:analysis_efficiency_switch}:
in Section~\ref{sec:accuracy} we show that our mixed precision Jacobi SVD
algorithm is backward stable, and is as accurate as the fixed precision Jacobi
SVD algorithm;
in Sections~\ref{sec:accuracy_switch} we analyze the quality of
conditioning in a lower precision;
in Section~\ref{sec:analysis_efficiency_switch} we discuss the accuracy of
different methods for switching precisions.
Numerical experiments are provided in Section~\ref{sec:experiments} to
demonstrate the efficiency of our mixed precision algorithm.

\section{The one-sided Jacobi SVD algorithm}
\label{sec:preliminary}
In this section, we briefly review the one-sided Jacobi SVD algorithm
developed by Drma\v{c} and Veseli\'{c} in~\cite{DV2008a,DV2008b}.
Compared with other variants of the Jacobi SVD algorithm, this one
is superior in performance while preserving high accuracy.
Implementation of this Jacobi SVD algorithm has been available in LAPACK
starting from versions~3.2.%
\footnote{Real versions were released in LAPACK version~3.2, while complex versions
were released later in LAPACK version~3.6.0.}

\subsection{Jacobi rotations}
\label{subsec:basic-algorithm}
For any column vectors \(x_1\), \(x_2\in\mathbb C^m\), there exists a unitary
matrix
\[
J=\bmat{c & s \\ -\conj s & c}
\]
with \(c\in\mathbb R\) such that the columns of \([x_1,x_2]J\) are orthogonal
to each other.%
\footnote{Actually, the choice of \(J\) is not unique for the purpose of
orthogonalizing \(x_1\) and \(x_2\).
The Jacobi SVD algorithm always chooses the one that is closest to the
identity matrix to ensure convergence.}
This orthogonal transformation plays a key role in the one-sided Jacobi SVD
algorithm.
To compute the SVD of a matrix \(X\in\mathbb C^{m\times n}\) with \(m\geq n\),
the one-sided Jacobi SVD algorithm repeatedly orthogonalizes a pair of columns of
\(X\) until \(X\herm X\) becomes numerically diagonal, i.e., all off-diagonal
entries of \(X\herm X\) have negligible magnitudes.
In other words, a sequence of matrices
\[
X^{(k)} = X^{(k-1)}J^{(k)}, \qquad \text{(\(k=1\), \(2\), \(\dotsc\), \(N\))},
\]
with \(X^{(0)}=X\) is computed until \(X^{(N)}\) has numerically orthogonal
columns after \(N\) steps, where the unitary matrix \(J^{(k)}\) is a so-called
\emph{Jacobi plane rotation matrix} of the form
\[
J^{(k)}=\bmat{I_{p_k-1} \\
& c_{k} & & s_{k} \\
& & I_{q_k-p_k-1} & \\
& -\conj s_{k} & & c_{k} \\
& & & & I_{n-q_k}}
\]
that orthogonalizes \(p_k\)-th and \(q_k\)-th columns of \(X^{(k-1)}\).
Then by normalizing the columns of \(X^{(N)}\) (i.e., \(X^{(N)}=U_X\Sigma\)
where \(U_X\herm U_X=I_m\) and \(\Sigma\) is diagonal and positive definite)
and setting \(V_X=J^{(1)}J^{(2)}\dotsm J^{(N)}\), we obtain
\(X=U_X\Sigma V_X\herm\) as the computed SVD of \(X\).%
\footnote{In practice, appropriate reordering of \(U_X\) and \(\Sigma\) is
required so that the diagonal entries of \(\Sigma\) are descending.}

Notice that the choice of \((p_k,q_k)\)'s largely affects the convergence of
the Jacobi SVD algorithm.
One popular strategy that guarantees the convergence is the \emph{cyclic}
strategy that periodically performs \emph{sweeps} of the following sequence of
\((p_k,q_k)\)'s:
\begin{equation}
\label{eq:cyclic}
(1,2),(1,3),\dotsc,(1,n),
\quad (2,3),\dotsc,(2,n),
\quad (3,4),\dotsc,(3,n),
\quad \dotsc,
\quad (n-1,n).
\end{equation}
Algorithm~\ref{alg:svj} summarizes the basic one-sided Jacobi SVD algorithm.

\begin{algorithm}[!tb]
\caption{One-sided Jacobi SVD algorithm (without preconditioning)}
\label{alg:svj}
\begin{algorithmic}[1]
\REQUIRE A matrix \(X\in\mathbb C^{m\times n}\),
the convergence threshold \(\epsilon_{\tol}\in(0,+\infty)\),
maximal iteration count \(\mathtt{maxiter}\).
\ENSURE The singular value decomposition \(X=U_X\Sigma V_X\herm\).

\STATE \(V_X\gets I\).
\label{line:first}
\STATE \(\mathtt{converged}\gets\mathtt{false}\), \(\mathtt{iter} \gets 1\).
\WHILE{\(\mathtt{iter}<\mathtt{maxiter}\) and \(\mathtt{converged}=\mathtt{false}\)}
    \STATE \(\mathtt{converged}\gets\mathtt{true}\).
    \FOR{each pair \((p,q)\) with \(p<q\) in~\eqref{eq:cyclic}}
        \IF{\(\lvert X(:,q)\herm X(:,p)\rvert>\epsilon_{\tol}
                \lVert X(:,p)\rVert_2\lVert X(:,q)\rVert_2\)}
            \STATE Choose a Jacobi rotation matrix \(J\) that
                orthogonalizes \(X(:,p)\) and \(X(:,q)\).
            \STATE \(X\gets XJ\).
            \STATE \(V_X \gets V_XJ\).
\label{line:accumulate}
            \STATE \(\mathtt{converged}\gets\mathtt{false}\).
        \ENDIF
    \ENDFOR
    \STATE \(\mathtt{iter}\gets \mathtt{iter}+1\).
\ENDWHILE
\label{line:Xinf}
\FOR{\(i\gets1:n\)}
    \STATE \(\Sigma(i,i)\gets\lVert X(:,i)\rVert_2\).
    \STATE \(U_X(:, i)\gets X(:,i)/\Sigma(i,i)\).
\ENDFOR
\STATE Sort the diagonal entries of \(\Sigma\), and permute the columns of
    \(U_X\) and \(V_X\) accordingly.
\end{algorithmic}
\end{algorithm}

The convergence of the one-sided Jacobi SVD algorithm can be accelerated by
de~Rijk's pivoting strategy~\cite{deRijk1989}.
This pivoting strategy splits each sweep into \(n-1\) subsweeps---in the
\(i\)-th subsweep, it identifies the column of \(X_{i:n}\) with largest
Euclidean norm and interchanges this dominant column with \(X_i\), and then
orthogonalizes the column pairs \((p,q)\) for \(q=p+1\), \(\dotsc\), \(n\).%
\footnote{Throughout the paper we adopt the MATLAB colon notation to denote
indices that describe submatrices.}

\subsection{Preconditioning}
\label{subsec:precond1}
In~\cite{DV2008a}, Drma\v{c} and Veseli\'{c} discussed a carefully designed
preconditioning strategy for the one-sided Jacobi SVD algorithm.
Roughly speaking, the key idea is to transform \(A\) to a ``more diagonally
dominant'' matrix \(X\), so that the column norms of \(X\) are good
approximations of the singular values.

First, a rank-revealing QR (RRQR) factorization \(AP=Q_1R\) is computed,
where \(Q_1\in\mathbb C^{m\times n}\) has orthonormal columns,
\(R\in\mathbb C^{n\times n}\) is upper triangular, and
\(P\in\mathbb R^{n\times n}\) is a permutation matrix.
If the columns of \(R\) are close to diagonally dominant, we set \(X=R\);
otherwise, an LQ factorization, \(R=LQ_2\), is performed, and we set \(X=L\).
Finally, the matrix \(X\) is fed to the one-sided Jacobi SVD algorithm.
The preconditioned algorithm is shown in Algorithm~\ref{alg:presvj}.

\begin{algorithm}[!tb]
\caption{Preconditioned one-sided Jacobi SVD algorithm}
\label{alg:presvj}
\begin{algorithmic}[1]
\REQUIRE A matrix \(A\in\mathbb C^{m\times n}\),
the convergence threshold \(\epsilon_{\tol}\in(0,+\infty)\),
maximal iteration count \(\mathtt{maxiter}\).
\ENSURE The full singular value decomposition \(A=U\Sigma V\herm\).

\STATE Factorize \(A\) such that \(AP=Q_1R\).
\IF{the \(R\) is almost diagonal}
    \STATE \(X\gets R\), \(V\gets P\).
\ELSE
    \STATE Factorize \(R\) such that \(R=LQ_2\).
    \STATE \(X\gets L\), \(V\gets PQ_2\herm\).
\ENDIF
\STATE Compute the SVD \(X=U_X\Sigma V_X\herm\) by Algorithm~\ref{alg:svj}.
\STATE \(U\gets Q_1\cdot U_X\), \(V\gets V\cdot V_X\).
\end{algorithmic}
\end{algorithm}

The preconditioning strategy has been observed to effectively accelerate the
convergence of the one-sided Jacobi SVD algorithm.
Moreover, it provides opportunities for fast implementations without
explicitly accumulating the Jacobi plane rotations.
For instance, 
if \(X\) is expressed as \(X=D_rZ_r\), where \(D_r\) is diagonal and \(Z_r\) has normalized rows, then \(V_X\) can be constructed using \(V_X=X^{-1}U_X\Sigma\) when \(Z_r\) is reasonably well-conditioned (this often occurs after preconditioning).
Consequently, Step~\ref{line:accumulate} in Algorithm~\ref{alg:svj} can be omitted.
This strategy improves the computational efficiency because accumulating
rotations is usually very time-consuming.

\section{A mixed precision Jacobi SVD algorithm}
\label{sec:algorithm}
In this section, we propose a mixed precision Jacobi SVD algorithm.
The basic idea is to perform SVD in a lower precision as a preconditioner, and
then refine the solution by the one-sided Jacobi SVD algorithm in working
precision.
This idea yields a basic mixed precision Jacobi SVD algorithm as shown in
Algorithm~\ref{alg:basic-msvj}.
However, we shall see in Section~\ref{subsec:experiments-simplemsvj}
that the performance of this simple algorithm is far from satisfactory.
The number of sweeps in Step~\ref{line:basic-msvj-refinement} is not much
lower than that of Algorithm~\ref{alg:presvj} so that only limited speedup can
be expected.

In order to achieve better performance, we need to redesign a mixed
precision algorithm carefully.
We split the algorithm into the following stages.
\begin{enumerate}
\item
QR preconditioning is performed in order to accelerate the
convergence and also avoid unnecessary complication for rectangular matrices.
\item
The singular value decomposition is computed in a lower precision.
\item
The solution in lower precision is transformed back to the working
precision.
\item
The one-sided Jacobi SVD algorithm without preconditioning (i.e.,
Algorithm~\ref{alg:svj}) is applied to refine the solution in working
precision.
\end{enumerate}
The first three stages can be interpreted as a careful preconditioning
strategy involving low precision arithmetic that produces a reasonably
good initial guess for Algorithm~\ref{alg:svj}.
If the initial guess is of good quality, the last stage starts with
nearly orthogonal columns and then quick convergence can be expected due to
the asymptotic quadratic convergence of the Jacobi algorithm.
Our mixed precision algorithm is briefly summarized as
Algorithm~\ref{alg:msvj}.
In the following we discuss the four stages of Algorithm~\ref{alg:msvj} in
details.
We use \(\double(\cdot)\) and \(\single(\cdot)\), respectively, to explicitly
convert the data into working and lower precisions.
In addition, \(\fl(\cdot)\) and \(\fllow(\cdot)\), respectively, denote the computed results in working and lower precisions.

\begin{algorithm}[!tb]
\caption{Basic mixed precision Jacobi SVD algorithm}
\label{alg:basic-msvj}
\begin{algorithmic}[1]
\REQUIRE A matrix $A\in\mathbb C^{m\times n}$ with \(m\geq n\),
the threshold $\tolalg\in(0,+\infty)$ for choosing the lower precision SVD
algorithm,
the orthogonality threshold $\tolorth\in(0,+\infty)$,
the condition number threshold $\tolcond\in(0,+\infty)$.
\ENSURE The full singular value decomposition $A=U\Sigma V\herm$.

\STATE Apply the preconditioned one-sided Jacobi algorithm (Algorithm~\ref{alg:presvj}) in lower precision to obtain \(\single(A)=U_\s \Sigma_\s V\herm_\s\).
\STATE Convert \(V_\s\) to the working precision and compute the QR factorization \(\double(V_\s)=QR\).
\STATE \(Y\gets AQ\).
\STATE Apply the one-sided Jacobi algorithm (Algorithm~\ref{alg:svj}) to \(Y\), i.e., \(Y=U\Sigma V_Y\herm\).
\label{line:basic-msvj-refinement}
\STATE \(V\gets V_YQ\).
\end{algorithmic}
\end{algorithm}

\begin{algorithm}[!tb]
\caption{Mixed precision Jacobi SVD algorithm}
\label{alg:msvj}
\begin{algorithmic}[1]
\REQUIRE A matrix $A\in\mathbb C^{m\times n}$ with \(m\geq n\),
the threshold $\tolalg\in(0,+\infty)$ for choosing the lower precision SVD
algorithm,
the orthogonality threshold $\tolorth\in(0,+\infty)$,
the condition number threshold $\tolcond\in(0,+\infty)$.
\ENSURE The full singular value decomposition $A=U\Sigma V\herm$.

\IF {$m > n$}
\label{line:preproc-start}
    \STATE Compute the QR factorization of $A$ such that $A=Q_0R_1$.
    \STATE $A_1\gets R_1$.
\ELSE
    \STATE $Q_0 \gets I$.
    \STATE $A_1\gets A$.
\ENDIF
\STATE Apply the QR preconditioning like Algorithm~\ref{alg:presvj} to $A_1$
       such that $X=Q_1\herm A_1PQ_2\herm$, where $Q_2$ could be the
       identity matrix.
\label{line:old-start}
\IF{$\texttt{cond}_R \leq \tol_{\rm{cond}}$ and many trailing columns of \(X\) have small norms}
\label{line:preproc-Q-start}
    \STATE $Q\gets I$. \label{line:msvj:notmp}
\ELSE
    \STATE $X_t\gets \single(X)D^{-1}$, where $D$ is a diagonal matrix and
           $D_{ii}$ is 2-norm of the $i$-th column of $\single(X)$. \label{line:msvj:test-hardcase2}
    \STATE $\mathtt{orth}\gets\lVert X_t\herm X_t-I\rVert_{\max}$.
    \IF{$\mathtt{orth}\leq\tolorth$}
        \STATE $Q\gets I$. \label{line:msvj:handle-hardcase2}
    \ELSE
        \IF{$\mathtt{orth}\leq\tolalg$}
\label{line:mp-start}
            \STATE Apply the one-sided Jacobi SVD algorithm
                   (Algorithm~\ref{alg:svj}) in lower precision to obtain
                   $\single(X)=U_{\s}\Sigma_{\s}V_{\s}\herm$ without
                   generating $V_{\s}$ explicitly.
        \ELSE
            \STATE Apply the QR SVD algorithm in lower precision to obtain
                   $\single(X)=U_{\s}\Sigma_{\s}V_{\s}\herm$ without
                   generating $V_{\s}$ explicitly.
        \ENDIF
        \STATE Convert \(U_{\s}\) back to the working precision and compute
               the QR factorization $X\herm\double(U_{\s})=QR_2$.
    \ENDIF
\ENDIF
\STATE $Y\gets XQ$.
\label{line:preproc-end}
\STATE Apply the one-sided Jacobi algorithm (Algorithm~\ref{alg:svj}) to
\(Y\), i.e., $Y=U_X\Sigma V_Y\herm$ ($V_Y$ is not explicitly generated unless $V$ is needed).
\label{line:main-start}
\STATE \(U \gets Q_0 Q_1 U_X\), and optionally \(V \gets P Q_2\herm V_Y Q\) if \(V\) is needed.
\label{line:main-end}
\end{algorithmic}
\end{algorithm}

\subsection{QR preconditioning}
\label{subsec:precond2}
As mentioned in Section~\ref{subsec:precond1}, preconditioning can simplify
the computation and accelerate the convergence of the one-sided Jacobi SVD
algorithm.
The preconditioning strategy in Algorithm~\ref{alg:msvj} is inherited from
that in~\cite{DV2008a}, with some minor changes.
We remark that this preconditioning strategy is very effective so that the
one-sided Jacobi SVD algorithm can often converge in a few sweeps.
As a result the benefit for computing a low precision SVD may be limited
in some special cases.
In Section~\ref{subsec:special-cases}, we shall present additional techniques
to handle these special cases.

When the input is a square matrix, we simply use the same preconditioning
strategy as in~\cite{DV2008a}.
For a rectangular matrix \(A\in\mathbb C^{m\times n}\) with \(m>n\), instead
of performing a rank-revealing QR factorization as in~\cite{DV2008a}, we
perform a plain QR factorization \(A=QR_1\) to obtain an \(n\times n\)
matrix~\(R_1\) and then apply the preconditioning strategy for square matrices
to \(R_1\).
Since RRQR is very costly, our revised strategy in general reduces the
computational time unless \(m\) is really close to \(n\).

The RRQR algorithm%
\footnote{To be precise, the RRQR algorithm here refers to \texttt{XGEQP3}
in LAPACK, which implements the Businger--Golub QR factorization with column
pivoting;
see~\cite{BG1965, DB2008}.}
can also be improved by using low precision
arithmetic.
Compared to the plain QR factorization, the RRQR algorithm requires a lot of
effort to determine the column permutation.
Hence pivoting is the dominant cost for RRQR in practice.
Once the permutation is already known, computing a plain QR factorization on a
permuted matrix is much cheaper than performing the RRQR algorithm.
In order to determine the permutation, it suffices to apply the RRQR algorithm
in a lower precision.
This leads to a mixed precision RRQR algorithm as shown in
Algorithm~\ref{alg:mprrqr}.
Though this algorithm in fact computes QR factorizations twice, it is still
observed to be faster than the usual RRQR algorithm in working precision
on the x86-64 architecture since the most expensive part is performed using a
lower precision.
We use this mixed precision algorithm to compute the RRQR factorization in the
preconditioning stage.

\begin{algorithm}[!tb]
\caption{Mixed precision RRQR algorithm}
\label{alg:mprrqr}
\begin{algorithmic}[1]
\REQUIRE A matrix $A\in\mathbb C^{m\times n}$.
\ENSURE The RRQR factorization $AP=QR$.

\STATE Convert \(A\) to a lower precision one:
\(A_{\s}\gets\mathtt{lower}(A)\).
\STATE Compute the RRQR factorization \(A_{\s}P=Q_{\s}R_{\s}\) in a lower
precision.
\STATE Use \(P\) to permute the columns of \(A\), i.e., \(A\gets AP\).
\STATE Compute the QR factorization (of the permuted matrix) \(A=QR\) in
working precision.
\end{algorithmic}
\end{algorithm}

\subsection{SVD in a lower precision}
After the first stage, we have already transformed \(A\) to a square matrix
\(X\) whose SVD is easier to compute.
The goal of the second stage in our algorithm is to compute the singular value
decomposition \(X=U_X\Sigma V_X\herm\) in a lower precision.

When computing the SVD in a lower precision, there are several possible
choices, such as Jacobi algorithm, QR algorithm, and divide and conquer
algorithm.
The divide and conquer algorithm is not a good option for highly accurate
computation~\cite{DV1992,Demmel1997,DV2008b} because we observed that the computed singular vectors are often not sufficiently
accurate so that many sweeps will be needed by the subsequent one-sided Jacobi
SVD algorithm after switching back to the working precision.
Though the QR algorithm is in general less accurate than the Jacobi
algorithm~\cite{DV1992}, Drma\v{c} pointed out in~\cite{D2017} that the QR
algorithm can compute the SVD accurately as long as it is carefully
preconditioned (using the strategy discussed in Section~\ref{subsec:precond2}).
Therefore we use either the Jacobi algorithm or the QR algorithm since they are
relatively accurate.
The choice mainly depends on the performance.

In order to determine which algorithm is used, we check the orthogonality of
\(X\) in the lower precision after QR preconditioning.
When the computed result \(\fllow(X)\) has nearly orthogonal columns (e.g.,
\(\lVert \fllow(X)\herm \fllow(X)-I\rVert=O(\machepslow^{1/2})\)), the Jacobi algorithm is chosen
as it converges very rapidly due to the quadratic convergence;
otherwise the QR algorithm is used because it is in general faster.

Finally, we remark that there is no need to compute both left and right
singular vectors of~\(X\) in this stage.
In the subsequent subsection we shall see that it suffices to compute \(U_X\)
only.

\subsection{Switching precisions}
\label{subsec:switch}
Once the singular value decomposition \(X=U_X\Sigma V_X\herm\) is
determined in a lower precision, it is necessary to transform \(V_X\)
into a unitary matrix \(Q\) in working precision.
This ensures that \(Y = XQ\) has nearly orthogonal columns, allowing it to be processed efficiently by the Jacobi SVD algorithm in the final stage.
We will outline the methods to build the unitary \(Q\) in working precision.

Note that directly converting the computed result of \(V_X\), i.e.,
\(\fllow(V_X)\), from lower precision to working precision is
problematic since the result is far from unitary in working precision,
i.e., \(\lVert \fllow(V_X)\herm \fllow(V_X)-I\rVert=O(\machepslow)
\gg\macheps\), where \(\macheps\) and \(\machepslow\), respectively, are the
unit round-off of working and lower precisions.
Hence a QR factorization of \(\fllow(V_X)=QR_2\) in working
precision is required to improve the orthogonality.
Then the matrix \(XQ\) can be fed to the one-sided Jacobi SVD algorithm in the
last stage.

There is an alternative approach to construct the unitary matrix \(Q\) in the
working precision.
Instead of working on the right singular vectors \(\fllow(V_X)\), we convert the computed left
singular vectors \(\fllow(U_X)\) to working precision.
Then a QR factorization \(X\herm \fllow(U_X)=QR_2\) is computed in working precision, so that theoretically the \(Q\)-factor of \(X\herm U_X\) is \(V_X\)
if computing \(U_X\), the product \(X\herm U_X\) and this QR factorization are all in exact arithmetic.
Finally, the matrix \(XQ\) is fed to the one-sided Jacobi SVD algorithm in the
last stage.
We prefer this approach in practice because forming \(V_X\) in the lower
precision is a very expensive operation.
For the one-sided Jacobi SVD algorithm, accumulating \(V_X\) is clearly very
costly;
for the QR SVD algorithm, generating \(V_X\) is observed to be much more
expensive than generating \(U_X\) in LAPACK's subroutine \texttt{xGESVD}.

\subsection{Refinement for nearly orthogonal matrices}
\label{subsec:refine}
After preconditioning and switching precisions, we eventually obtain \(\fl(Y)\approx U_X\Sigma\) with nearly orthogonal columns, while
columns of \(\fl(Y)\) are orthogonal at the level of \(\machepslow\).
In this stage, we apply the one-sided Jacobi SVD algorithm on \(\fl(Y)\) to
compute the SVD in the working precision, because the Jacobi SVD algorithm can take
advantage of the property that \(\fl(Y)\) has nearly orthogonal columns.
Typically it takes three sweeps for our mixed precision algorithm to refine
the solution if \(\machepslow\) is close to \(\macheps^{1/2}\) (which holds
for IEEE single and double precisions).

After switching precisions, trailing columns of \(\fl(Y)\), that correspond to
small singular values, can occasionally be less accurate compared to leading
columns that correspond to large singular values.
We test the orthogonality of \(\fl(Y)\) and partition \(\fl(Y)\) into two blocks
\(\fl(Y)=[Y_1,Y_2]\) so that \(\lVert Y_1\herm Y_1-I\rVert=O(\machepslow)\) while
\(\max\lbrace\lVert Y_2\herm Y_1\rVert,\lVert Y_2\herm Y_2-I\rVert\rbrace\gg\machepslow\).
If the \(Y_2\) block is non-empty, we apply Jacobi rotations in \(Y_2\) first,
and then apply Jacobi rotations between \(Y_1\) and \(Y_2\) in one sweep.
Thanks to careful QR preconditioning, the solution from lower precision
usually has good relative accuracy.
We have observed that in most cases the \(Y_2\) block is empty, i.e., the
orthogonality of \(Y\) is satisfactory.
Our strategy for handling \(Y_2\) is only a safeguard, and is rarely
activated.
We remark that in the final stage it is safe to skip several techniques of the 
one-sided Jacobi SVD algorithm proposed in~\cite{DV2008b} that aim at handling
difficult cases as the solution from the lower precision is reasonably
accurate.

The computation of right singular vectors \(V_X\) is also adjusted compared to
the algorithm in~\cite{DV2008b}.
When the problem is reasonably well-conditioned, \(V_X\) can be formed through
\(V_X=\Sigma^{-1}U_X\herm X\);
otherwise, \(V_X\) is computed by accumulating Jacobi rotations, which is
expensive.
In the latter case, instead of directly accumulating Jacobi rotations, we
first compute \(U_X\) without accumulating~\(V_X\).
Then similar to the strategy of switching precisions, we form \(V_X\) from the
QR factorization of \(X\herm U_X\).
These left and right singular vectors are used as the initial guess of another
round of the one-sided Jacobi SVD algorithm by accumulating Jacobi rotations.
Since such an initial guess is highly accurate, the final Jacobi SVD algorithm
always converges in one sweep.
The cost of computing \(V_X\) is much lower compared to that by the original
strategy in~\cite{DV2008b} since we avoid accumulating Jacobi rotations for
many sweeps.

\subsection{Handling special cases}
\label{subsec:special-cases}
The purpose of computing SVD in a lower precision is to reduce the number of
sweeps required by the one-sided Jacobi SVD algorithm in working precision.
As we mentioned in Section~\ref{subsec:refine}, the accuracy of the solution
from lower precision is at the level of \(O(\machepslow)\), and it still
requires about three sweeps to refine the solution in working precision.
However, there exist matrices for which the benefit of computing SVD in a lower
precision is limited.
For these matrices our mixed precision Jacobi SVD algorithm can be even slower
than the fixed precision algorithm.
Since the ultimate purpose of developing a mixed precision Jacobi SVD
algorithm is to improve the computational efficiency, it is necessary to
detect this type of matrices and avoid the overhead caused by the additional work
performed in lower precision.
In the following, we identify three special cases in which computing SVD in
lower precision needs to be avoided.

One special case can be easily detected after QR preconditioning.
After computing the RRQR factorization of \(A\), if the condition number of
\(R\), denoted by \(\texttt{cond}_R\) in Algorithm~\ref{alg:msvj}, is relatively small (e.g., \(\kappa(R)\leq\tolcond\) with
\(\tolcond=O(n^{1/4})\) based on our experience of numerical experiments), QR preconditioning often leads to quick convergence
of the one-sided Jacobi SVD algorithm in working precision, and hence we
do not compute the SVD in lower precision;
see Lines~\ref{line:preproc-Q-start}--\ref{line:msvj:notmp} in Algorithm~\ref{alg:msvj}.

Another special case occurs when the columns of \(X\) are almost
orthogonal in lower precision after QR preconditioning.
In this case, computing SVD in lower precision becomes totally needless
because it cannot further accelerate the convergence of the Jacobi SVD
algorithm in working precision.
To detect this case, we first scale the columns of \(X\) so that they have
unit column norms.
Then we check whether \(\lVert X\herm X-I\rVert_{\max}\) is below a prescribed
threshold \(\tolorth\),%
\footnote{For instance, for \(\machepslow=2^{-24}\approx6\times10^{-8}\), we
can choose \(\tolorth\approx10^{-5}\).}
where \(\lVert\cdot\rVert_{\max}\) represents the largest absolute value of
matrix elements.
If \(\lVert X\herm X-I\rVert_{\max}\leq\tolorth\) is detected, we simply skip
the SVD in lower precision and directly use the one-sided Jacobi SVD
algorithm in working precision to solve this problem;
see Lines~\ref{line:msvj:test-hardcase2}--\ref{line:msvj:handle-hardcase2} in Algorithm~\ref{alg:msvj}.
The test is relatively cheap compared to the entire SVD algorithm since it
only requires one matrix--matrix multiplication in lower precision.

The third special case often (but not always) occurs for columnwise graded
matrices of the form \(A = BD\), where \(D\) is an ill-conditioned diagonal
matrix.
When most singular values of~\(A\) are relatively small, it is not appropriate
to compute SVD in lower precision as the benefit is limited.
It is often possible to detect this case after QR preconditioning:
if many trailing columns of \(X\) have small norms,%
\footnote{A heuristic criterion is to check \(1/4\sim1/3\) of the columns.}
then the trailing singular values are also small; see Line~\ref{line:preproc-Q-start} in Algorithm~\ref{alg:msvj}.
In this case we skip the SVD in lower precision and move on to the
one-sided Jacobi SVD algorithm in working precision.

\section{Accuracy of the computed SVD}
\label{sec:accuracy}
In this section, we show that Algorithm~\ref{alg:msvj} is numerically stable
in the sense that it inherits nice high accuracy properties from usual
one-sided Jacobi SVD algorithms.
We use \(\,\tilde{}\,\) and \(\,\hat{}\,\) to denote theoretical and
computed results, respectively.

In the following we make two assumptions on orthogonal transformations.
We first assume that a computed QR factorization (possibly with column
exchange) is \emph{columnwise} backward stable, in the sense that
\begin{equation}
\label{eq:RRQR-error}
(A+\Delta A)P=\tilde Q_0\hat R_0,
\qquad
\lVert\Delta A(:, i)\rVert_2\leq\epsrrqr\lVert A(:,i)\rVert_2
\end{equation}
holds for each \(i\), where \(\epsrrqr\) is bounded by a low degree polynomial
of the matrix dimension times the unit round-off~\(\macheps\).
We also assume that steps~\ref{line:first}--\ref{line:Xinf} of
Algorithm~\ref{alg:svj} produces
\[
\hat X_{\infty}=(X+\Delta X_0)\tilde V_X
\]
with a \emph{rowwise} backward error
\[
\lVert\Delta X_0(i, :)\rVert_2\leq\epsJh\lVert X(i,:)\rVert_2,
\]
where \(\epsJh=\bigO(n\macheps)\).
These assumptions are plausible in practice because applying orthogonal
transformations from the left\slash{}right only introduces
columnwise\slash{}rowwise backward error in floating-point arithmetic.
The same notation and assumptions are used by the analysis in~\cite{DV2008a}.
Another constant \(\epsJbar=\epsJh+\macheps(1+\epsJh)\) is also defined to
simplify the notation.

Steps~\ref{line:preproc-start}--\ref{line:preproc-end} of
Algorithm~\ref{alg:msvj} produces an initial guess, which will be iteratively
refined by the one-sided Jacobi SVD algorithm in working precision
(i.e., steps~\ref{line:main-start}--\ref{line:main-end}).
The accuracy of the final solution is mainly determined by the Jacobi SVD
algorithm in working precision, as long as the initial guess is
reasonably accurate.
Since the numerical stability of steps~\ref{line:main-start}--%
\ref{line:main-end} has already been analyzed in~\cite{DV2008a}, we shall
only focus on the difference between our algorithm and that in~\cite{DV2008a}.

Our preconditioning strategy is similar to that in~\cite{DV2008a}.
For rectangular matrices, Algorithm~\ref{alg:msvj} performs an extra QR
factorization in the preprocessing stage.
Theoretically, this additional QR factorization can be merged with the
subsequent RRQR factorization, because \(A=Q_1R_1\) and \(R_1P=Q_2R_2\) can be
interpreted as \(AP=(Q_1Q_2)R_2\).
According to~\eqref{eq:RRQR-error}, this RRQR factorization is columnwise
stable.%
\footnote{The constance \(\epsqr\) in~\eqref{eq:RRQR-error} needs to be chosen
slightly larger than its tightest estimate so that~\eqref{eq:RRQR-error} holds
even when an extra QR factorization is performed.}
In~\cite{DV2008a}, the analysis for preconditioning remains valid since it
merely assumes the columnwise stability of the RRQR factorization.

The analysis of~\cite{DV2008a} is based on the fact that applying the one-sided Jacobi
SVD algorithm \emph{without} preconditioning (i.e.,
\texttt{DGESVJ}\slash\texttt{ZGESVJ}) to \(X\) produces a rowwise stable solution;
see~\cite[Proposition~4.1]{DV2008a}.
The major difference between our algorithm and that in~\cite{DV2008a} is that
we apply \texttt{DGESVJ}\slash\texttt{ZGESVJ} (assuming the working precision is
double precision) to the matrix \(XQ\) instead of \(X\).
Notice that the multiplication of \(X\) and any unitary matrix can be computed in a rowwise stable manner.
Naturally we can expect that applying the one-sided Jacobi SVD algorithm without preconditioning to the computed result \(XQ\) is also rowwise stable,
where \(Q\) can be any unitary matrix.
To show the stability of Algorithm~\ref{alg:msvj}, it suffices to establish
the following Proposition~\ref{prop:dsvj}, which is a revised version
of~\cite[Proposition~4.1]{DV2008a} tailored to our algorithm.

\begin{proposition}
\label{prop:dsvj}
Let $\hat{U}_X$ and $\hat{\Sigma}$, respectively, consist of the computed
left singular vectors and singular values of the matrix $\fl(X\hat{Q})$.
Then there exists a unitary matrix $\Tilde{V}_{X}$ and a backward perturbation
\(F\) such that
\[
X + F = \hat{U}_X\hat{\Sigma}\Tilde{V}_{X}\herm,
\]
where
\begin{equation}
\label{eq:rowwise-bound}
\lVert F(i,:)\rVert_2\leq(\epsJbar(1+\epsrrqr)+\epsrrqr)\lVert X(i,:)\rVert_2
\end{equation}
for all $i$.
In addition, if \(\hat V_X\) consists of the computed right singular vectors,
then \(\hat{U}_X\hat{\Sigma}\hat{V}_{X}\herm = (X + F)(I + E_0\herm)\),
where \(\lVert E_0\rVert_2\leq \sqrt{n}\,\epsJbar\).
\end{proposition}

\begin{proof}
We only prove~\eqref{eq:rowwise-bound} since the rest is easy to verify.

There exists a unitary matrix $\Tilde{Q}$ and a small perturbation
$\Delta (X\herm U_{\s})$ such that
\[
X\herm\double(U_{\s}) + \Delta (X\herm U_{\s})
= \Tilde{Q}\hat{R}_t.
\]
Then the computed product satisfies
$\fl(X\hat{Q}) = (X + \Delta X_0)\Tilde{Q}$,
where
\begin{equation} \label{eq:prop-proof:normDX0}
\begin{split}
\lVert \Delta X_0(i, :)\rVert_2\leq \epsrrqr\lVert X(i, :)\rVert_2.
\end{split}
\end{equation}
By~\cite[Proposition~4.1]{DV2008a}, there exists a unitary matrix
$\Tilde{V}_{XQ}$ and a perturbation $F_{XQ}$ such that
\[
\fl(X\hat{Q}) + F_{XQ} =
(X + \Delta X_0)\hat{Q} + F_{XQ} =
\hat{U}_X\hat{\Sigma}\Tilde{V}_{XQ}\herm
\]
and \(\lVert F_{XQ}(i, :)\rVert_2\leq
\epsJbar\lVert\fl(X\hat{Q})(i, :)\rVert_2\).
The unitary invariance of the \(2\)-norm implies
\begin{equation} \label{eq:prop-proof:normFQ}
\begin{split}
\lVert (F_{XQ}\Tilde{Q}\herm)(i,:)\rVert_2&=\lVert F_{XQ}(i,:)\rVert_2\\
&\leq\epsJbar\lVert\fl(X\hat{Q})(i,:)\rVert_2\\
&\leq\epsJbar\bigl(\lVert X(i,:)\rVert_2+\epsrrqr\lVert X(i,:)\rVert_2\bigr)\\
&\leq\epsJbar(1+\epsrrqr)\lVert X(i,:)\rVert_2
\end{split}
\end{equation}
for all \(i\).
Moreover, we have
\[
X + \Delta X_0 + F_{XQ}\Tilde{Q}\herm = \hat{U}_X\hat{\Sigma}\Tilde{V}_{X}\herm,
\]
where $\Tilde{V}_{X} = \Tilde{Q}\Tilde{V}_{XQ}$.
Let $F = \Delta X_0 + F_{XQ}\Tilde{Q}\herm$.
We conclude that~\eqref{eq:rowwise-bound} holds because~\eqref{eq:prop-proof:normDX0} and~\eqref{eq:prop-proof:normFQ} imply that
\[
\lVert F(i,:)\rVert_2 \leq \lVert \Delta X_0(i, :)\rVert_2 + \lVert (F_{XQ}\Tilde{Q}\herm)(i,:)\rVert_2 \leq \bigl(\epsrrqr + \epsJbar(1+\epsrrqr)\bigr)\lVert X(i,:)\rVert_2.
\qedhere
\]
\end{proof}

By replacing~\cite[Proposition~4.1]{DV2008a} with Proposition~\ref{prop:dsvj},
Algorithm~\ref{alg:msvj} fits into the framework of~\cite{DV2008a}---the rest
of the analysis in~\cite[Section~5]{DV2008a} remains valid.
Therefore, Algorithm~\ref{alg:msvj} also possesses nice high accuracy
properties \emph{in working precision}, though a lot of computations are
performed in lower precision.

However, note that this analysis can be valid for any unitary matrix \(Q\) and cannot illustrate why the one-sided Jacobi SVD algorithm without preconditioning can converge such quickly when using SVD in lower precision as the preconditioner.
Thus, in the next section, we analyze the efficiency of our new preconditioning strategy and show the advantage of using our special \(Q\) computed by~\ref{line:preproc-Q-start}--\ref{line:preproc-end} in Algorithm~\ref{alg:msvj}.

\section{Efficiency of the preconditioning}
\label{sec:accuracy_switch}
In Algorithm~\ref{alg:msvj}, after Step~\ref{line:old-start}, we obtain the
matrix \(X\).
Our new preconditioning strategy is as the following procedure:
\begin{align*}
& U\Sigma V\herm \gets X,\\
& B \gets X\herm U,\\
& QR \gets B,\\
& Y \gets XQ,
\end{align*}
where \(U\), \(V\), \(Q\in \mathbb{C}^{n\times n}\) are unitary matrices, and
the diagonal matrix \(\Sigma = \diag(\sigma_1, \sigma_2, \dotsc, \sigma_n)\)
satisfies \(\sigma_1\geq\sigma_2\geq\dotsb\geq\sigma_n\).
Note that only the first step is computed in lower precision.

Let \(\hat{U}\) and \(\hat{\Sigma}\) denote the computed result using
floating-point arithmetic by SVD of \(X\);
\(\hat{Q}\) and~\(\hat{R}\) denote the computed results using
floating-point arithmetic by the QR factorization of \(\fl(X\herm
\hat{U})\).
Taking rounding errors into account, the computed quantities satisfy
\begin{align}
\label{eq:pre_X}
& X = \hat{U}\hat{\Sigma} \Tilde{V}\herm + \Delta X\quad \text{with}\quad \lVert \Delta X\rVert_\fro\leq \epsSVD\lVert X\rVert_\fro\quad \text{and}\quad \lVert\hat{U}\herm \hat{U}-I\rVert_\fro\leq\epsU,\\
\label{eq:pre_B}
& \hat{B} := \fl(X\herm \hat{U}) = X\herm \hat{U} + \Delta(X\herm U)\quad \text{with}\quad \lVert \Delta(X\herm U)\rVert_\fro\leq \sqrt{n}\,\gamma_n\lVert\hat{U}\rVert_2\lVert X\rVert_\fro,\\
\label{eq:pre_Q}
& \hat{B} + \Delta B = \Tilde{Q}\hat{R}\quad \text{with}\quad \lVert \Delta B\rVert_\fro\leq \epsqr\lVert \hat{B}\rVert_\fro\quad \text{and}\quad \lVert\hat{Q}-\Tilde{Q}\lVert_\fro\leq\epsQ,\\
\label{eq:pre_Y}
& \hat{Y} := \fl(X\hat{Q})=X\hat{Q}+\Delta(X\hat{Q})\quad \text{with}\quad \lVert \Delta(X\hat{Q})\rVert_\fro\leq \sqrt{n}\,\gamma_n\lVert \hat{Q}\rVert_2\lVert X\rVert_\fro,
\end{align}
where \(\Tilde{V}\) and \(\Tilde{Q}\) are unitary matrices.
Moreover, if we use the one-sided Jacobi algorithm to compute SVD of X in lower precision,
\eqref{eq:pre_X} can be improved as, from~\cite[Proposition~4.1]{DV2008a},
\begin{align}
\label{eq:pre_XJacobi}
& X = \hat{U}\hat{\Sigma} \Tilde{V}\herm + \Delta X\quad \text{with}\quad \lVert \Delta X(i, :)\rVert_2\leq \epsSVD\lVert X(i, :)\rVert_2\quad \text{and}\quad \lVert\hat{U}\herm \hat{U}-I\rVert_\fro\leq\epsU.
\end{align}
In the above equations~\eqref{eq:pre_X}--\eqref{eq:pre_XJacobi}, \(\epsSVD\) and \(\epsU\) are functions of \(n\) and \(\machepslow\); \(\gamma_n:=n\macheps/(1-n\macheps)\), \(\epsqr\) and \(\epsQ\) are functions of \(n\) and \(\macheps\).
Equations~\eqref{eq:pre_B} and~\eqref{eq:pre_Y} follow from~\cite[Equation~3.13]{H2002} because
\begin{equation}
\lVert \Delta(X\herm U)\rVert_\fro\leq \gamma_n\lVert\hat{U}\rVert_\fro\lVert X\rVert_\fro
\leq \sqrt{n}\,\gamma_n\lVert\hat{U}\rVert_2\lVert X\rVert_\fro
\end{equation}
and it holds for \(\lVert \Delta(X\hat{Q})\rVert_\fro\) similarly.
Equation~\eqref{eq:pre_Q} is from~\cite[Theorem~19.4 and Equation~(19.13)]{H2002}.
From~\eqref{eq:pre_X}, we easily obtain
\begin{equation}
\label{eq:pre_U}
\lVert \hat{U}\rVert_2\leq 1+\epsU,\quad
\lVert \hat{U}\herm\hat{U}\rVert_2\leq 1+\epsU.
\end{equation}

In the refinement stage of our algorithm, we apply the one-sided Jacobi algorithm to the matrix~\(\hat{Y}\).
Thus, the speed of convergence of refinement depends on the orthogonality of \(\hat{Y}\),
i.e., when the columns of \(\hat{Y}\) are closer to orthogonal, the convergence of the Jacobi algorithm is faster.
Thus, we establish Theorem~\ref{thm:pre} to clarify how close \(\hat{Y}\) is
to orthogonality and to show the efficiency of our preconditioning.

\begin{theorem}
\label{thm:pre}
Assume that \eqref{eq:pre_X}--\eqref{eq:pre_XJacobi} hold together with
\begin{equation}
\label{eq:pre_assump}
\epsU\lVert\hat{\Sigma}\rVert_2\lVert\hat{\Sigma}^{-1}\rVert_2
+\sqrt{n}\,\bigl(\epsSVD+\sqrt{n}\,\gamma_n+\epsqr+\sqrt{n}\,\gamma_n\epsqr\bigr)\bigl(1+\epsU\bigr)
\lVert X\rVert_2\lVert\hat{\Sigma}^{-1}\rVert_2
<\sqrt{\frac32}-1.
\end{equation}
Then the matrix \(\hat{Y}\) obtained by~\eqref{eq:pre_X}--\eqref{eq:pre_Y} satisfies the following inequality:
\begin{align}
&\off(\hat{\Sigma}^{-1}\hat{Y}\herm\hat{Y}\hat{\Sigma}^{-1})\nonumber\\
:={}&\left(\sum_{i=1}^n\sum_{j=1, j\neq i}^n \lvert (\hat{\Sigma}^{-1}\hat{Y}
\herm\hat{Y}\hat{\Sigma}^{-1})_{ij}\rvert^2\right)^{1/2}\nonumber\\
\leq{}& \frac{2(\sqrt{2}+\sqrt{3}\,)+3}{1-\epsU}\epsU
+\frac{2(\sqrt{2}+\sqrt{3}\,)(1+\epsU)}{1-\epsU}\sqrt{n}\,\epsSVD\lVert
X\rVert_2\lVert\hat{\Sigma}^{-1}\rVert_2\nonumber\\
&+O(\macheps+\machepslow^2)\lVert X\rVert_2\lVert\hat{\Sigma}^{-1}\rVert_2.
\label{eq:pre-weak}
\end{align}

Moreover, if~\eqref{eq:pre_XJacobi} holds, then
\begin{align}
\off(\hat{\Sigma}^{-1}\hat{Y}\herm\hat{Y}\hat{\Sigma}^{-1})
\leq{}& \frac{2(\sqrt{2}+\sqrt{3}\,)+3}{1-\epsU}\epsU
+\frac{2(\sqrt{2}+\sqrt{3}\,)(1+\epsU)}{1-\epsU}\sqrt{n}\,\epsSVD\lVert Z_r^{-1}\rVert_2
\nonumber\\
&+O(\macheps+\machepslow^2)\lVert X\rVert_2\lVert\hat{\Sigma}^{-1}\rVert_2
\label{eq:pre-strong}
\end{align}
with \(X = D_rZ_r\), where \(D_r\) is a diagonal matrix and each row of \(Z_r\)
is a unit vector, i.e., \(Z_r\) consists of normalized rows of \(X\).
\end{theorem}

\begin{remark}
As shown in~\cite[Proposition~4.1]{DV2008a}, the computed SVD result of the Jacobi algorithm
has rowwise backward stability.
This implies that its accuracy depends primarily on the relative condition number
of \(Z_r\) and maintains the validity of~\eqref{eq:pre_XJacobi}.
Conversely, for the divide and conquer algorithm, compliance with \eqref{eq:pre_XJacobi} is typically not achieved.
Note that~\eqref{eq:pre-strong} necessitates~\eqref{eq:pre_XJacobi} in Theorem~\ref{thm:pre}.
Thus, the upper bound of \(\off(\hat{\Sigma}^{-1}\hat{Y}\herm\hat{Y}\hat{\Sigma}^{-1})\)
is primarily influenced by the condition number of \(X\) if we use the divide and conquer algorithm.
If we use the Jacobi algorithm, it mainly depends on the relative condition number
of \(Z_r\), which is usually much smaller than the condition number of \(X\).
Thus, the matrix \(Q\) computed by the Jacobi algorithm is more accurate.
This is the reason why we do not use the divide and conquer algorithm in our
mixed precision algorithm.
\end{remark}

To prove Theorem~\ref{thm:pre}, we first establish Lemma~\ref{lemma:hatY} to
reformulate~\eqref{eq:pre_Y}.
We use \(\up(\cdot)\) and \(\low(\cdot)\), respectively, to denote the strict
upper and lower triangular parts plus half of the diagonal part of the matrix.

\begin{lemma}
\label{lemma:hatY}
Assume that \eqref{eq:pre_X}--\eqref{eq:pre_Y} hold.
Let \(E_1=\Delta(X\herm U)+\Delta B\) and \(E_Y=-\hat{U}\iherm
E_1\herm\Tilde{Q}+X(\hat{Q}-\Tilde{Q})+\Delta(X\hat{Q})\).
Then
\begin{equation}
\label{eq:pre_finalY}
\hat{Y} = \hat{U}\iherm\hat{R}\herm+E_Y
\end{equation}
with
\begin{equation}
\label{eq:pre_normE1EY}
\begin{aligned}
\lVert E_1\rVert_\fro\leq \sqrt{n}\,\bigl(\sqrt{n}\,\gamma_n(1+\epsqr)+\epsqr\bigr)(1+\epsU)\lVert X\rVert_2
\quad\text{and}\quad\lVert E_Y\rVert_\fro
&\leq O(\macheps)\lVert X\rVert_2,
\end{aligned}
\end{equation}
where \(\hat{R}\) satisfies
\begin{subequations}
\begin{align}
\Tilde{Q}\hat{R}
&=\Tilde{V}\hat{\Sigma}+\Tilde{V}\hat{\Sigma}E_U+\Delta X\herm\hat{U}+E_1\label{eq:pre_QR1-0}\\
&=\Tilde{V}\hat{\Sigma}\bigl(I+\up(E_U)\bigr)+\Tilde{V}\hat{\Sigma}\low(E_U)+\Delta X\herm\hat{U}+E_1.\label{eq:pre_QR1-1}
\end{align}
\end{subequations}
\end{lemma}

\begin{proof}
By~\eqref{eq:pre_Y}, we have
\begin{equation}
\label{eq:pre_middleY}
\begin{aligned}
\hat{Y} &= X\Tilde{Q}+X(\hat{Q}-\Tilde{Q})+\Delta(X\hat{Q}).
\end{aligned}
\end{equation}
Here \(\Tilde{Q}\) is computed by the QR factorization of \(\hat{B}\), i.e.,
\begin{equation}
\label{eq:pre_QR0}
X\herm\hat{U}+E_1=\Tilde{Q}\hat{R},
\end{equation}
and \(E_1\) satisfies
\begin{equation}
\label{eq:pre_E1}
\begin{aligned}
\lVert E_1\rVert_F
&\leq \lVert\Delta(X\herm U)\rVert_\fro
+\lVert\Delta B\rVert_\fro\\
&\leq \lVert\Delta(X\herm U)\rVert_\fro
+\epsqr\bigl(\lVert X\herm \hat{U}\rVert_\fro+\lVert\Delta(X\herm U)\lVert_\fro\bigr)\\
&\leq \sqrt{n}\,\bigl(\sqrt{n}\,\gamma_n(1+\epsqr)+\epsqr\bigr)(1+\epsU)\lVert X\rVert_2.
\end{aligned}
\end{equation}
Furthermore, by~\eqref{eq:pre_QR0}, we have
\begin{equation*}
X\Tilde{Q}=\hat{U}\iherm\hat{R}\herm-\hat{U}\iherm E_1\herm\Tilde{Q}.
\end{equation*}
Combined with~\eqref{eq:pre_middleY}, we obtain~\eqref{eq:pre_finalY}
and~\eqref{eq:pre_normE1EY} because
\begin{equation}
\label{eq:pre_normEY}
\begin{aligned}
\lVert E_Y\rVert_\fro
&\leq (1+\epsU)\lVert E_1\rVert_\fro+\epsQ\lVert X\rVert_2+\sqrt{n}\,\gamma_n\lVert X\rVert_2\lVert \hat{Q}\rVert_\fro
\leq O(\macheps)\lVert X\rVert_2.
\end{aligned}
\end{equation}

Note that \(\hat{R}\) is the computed \(R\)-factor in the QR factorization of
\(X\herm\hat{U}+E_1\) by~\eqref{eq:pre_QR0}.
Then
\begin{align*}
\Tilde{Q}\hat{R}
&=X\herm\hat{U}+E_1\\
&=\Tilde{V}\hat{\Sigma}+\Tilde{V}\hat{\Sigma}E_U+\Delta X\herm\hat{U}+E_1\\
&=\Tilde{V}\hat{\Sigma}\bigl(I+\up(E_U)\bigr)+\Tilde{V}\hat{\Sigma}\low(E_U)+\Delta X\herm\hat{U}+E_1.
\qedhere
\end{align*}
\end{proof}

From the proof of Lemma~\ref{lemma:hatY}, the matrix \(\hat{R}\) is also the
\(R\)-factor in the QR factorization of
\(\Tilde{V}\herm\bigl(X\herm\hat{U}+E_1\bigr)\).
Then, according to~\eqref{eq:pre_QR1-0}, we have
\begin{equation}
\label{eq:pre_QR2}
\Tilde{V}\herm \Tilde{Q}\hat{R}=\hat{\Sigma}+F_1
\end{equation}
with \(F_1=\hat{\Sigma}E_U+\Tilde{V}\herm \Delta X\herm\hat{U}
+\Tilde{V}\herm E_1\).
Equation~\eqref{eq:pre_QR2} can be viewed as a perturbed QR factorization.
The following Lemma~\ref{lemma:perturQR}, which is derived
from~\cite[Theorem~6.1]{CP2001}, provides a useful tool to handle the QR
factorization of a perturbed matrix.

\begin{lemma}
\label{lemma:perturQR}
Let \(X\in\mathbb{R}^{n\times n}\) be of full column rank.
The QR factorization of \(X\) is \(X=QR\).
Suppose that \(F\in\mathbb{R}^{n\times n}\) satisfies
\(\bigl\lVert \lvert X\rvert\cdot\lvert F\rvert\bigr\rVert_2< 1\) and
\(\lVert FR^{-1} \rVert<\sqrt{3/2}-1\).
Then \(X+F\) has the unique QR factorization
\[
X+F=\bigl(Q+\Delta Q\bigr)\bigl(R+\Delta R\bigr)
\]
with
\[
\lVert \Delta RR^{-1}\rVert_\fro\leq (\sqrt{2}+\sqrt{3}\,)\lVert FR^{-1}\rVert_\fro.
\]

Moreover, if \(Q=I\) and \(RR_D^{-1}\) is diagonal, then
\[
\lVert \Delta R\rVert_\fro\leq 8\lVert F\rVert_\fro\lVert R_D^{-1}\rVert_2\lVert R_D\rVert_2.
\]
\end{lemma}

\begin{proof}
The first statement is a direct consequence of~\cite[Theorem~6.1]{CP2001}.
Thus, we only need to prove the second one.

Let \(D=RR_D^{-1}\).
According to~\cite[Section~6]{CP2001}, we have
\begin{equation*}
\Delta RR_D^{-1}=\up\bigl(FR_D^{-1}+D^{-1}R_D\iherm F\herm D +D^{-1}R_D\iherm(F\herm F-\Delta R\herm \Delta R)R_D^{-1}\bigr).
\end{equation*}
Combined with the first statement, it holds that
\begin{align*}
&\lVert\Delta RR_D^{-1}\rVert_\fro\\
\leq{}&\lVert FR_D^{-1}\rVert_\fro+\lVert D^{-1}\up(R_D\iherm F\herm)D\rVert_\fro
+\lVert FR_D^{-1}D^{-1}\rVert_\fro\lVert FR_D^{-1}\rVert_\fro
+\lVert \Delta RR_D^{-1}D^{-1}\rVert_\fro\lVert \Delta RR_D^{-1}\rVert_\fro\\
\leq{}& 2\lVert FR_D^{-1}\rVert_\fro
+\lVert FR_D^{-1}D^{-1}\rVert_\fro\lVert FR_D^{-1}\rVert_\fro
+(\sqrt{2}+\sqrt{3}\,)\lVert FR_D^{-1}D^{-1}\rVert_\fro\lVert \Delta RR_D^{-1}\rVert_\fro.
\end{align*}
Under the assumption \(\lVert FR^{-1}\rVert_\fro\leq\sqrt{3/2}-1\),
\(\lVert\Delta RR_D^{-1}\rVert_\fro\) can be bounded by
\begin{equation*}
\lVert\Delta RR_D^{-1}\rVert_\fro
\leq \frac{2+\lVert FR^{-1}\rVert_\fro}{1-(\sqrt{2}+\sqrt{3}\,)\lVert FR^{-1}\rVert_\fro}\lVert FR_D^{-1}\rVert_\fro
< 8\lVert FR_D^{-1}\rVert_\fro
\leq 8\lVert F\rVert_\fro\lVert R_D^{-1}\rVert_2
\end{equation*}
and \(\lVert\Delta R\rVert_\fro\leq\lVert\Delta RR_D^{-1}\rVert_\fro\lVert R_D\rVert_2\).
This completes the proof.
\end{proof}


With the aid of Lemmas~\ref{lemma:hatY} and~\ref{lemma:perturQR}, we are ready
to prove Theorem~\ref{thm:pre}.
\begin{proof}[Proof of Theorem~\ref{thm:pre}]
By~\eqref{eq:pre_QR1-1}, we have
\begin{equation}
\label{eq:pre_QR3}
\begin{aligned}
\hat{\Sigma}\bigl(I+\up(E_U)\bigr)+\hat{\Sigma}\low(E_U)+\Tilde{V}\herm\Delta X\herm\hat{U}+\Tilde{V}\herm E_1
&=\Tilde{V}\herm\Tilde{Q}\hat{R}
=(Q+\Delta Q)(R+\Delta R)
\end{aligned}
\end{equation}
with \(Q=I\) and \(R=\hat{\Sigma}\bigl(I+\up(E_U)\bigr)\).
Similarly to~\eqref{eq:pre_offY}, we conclude by~\eqref{eq:pre_finalY}
and~\eqref{eq:pre_QR3} that
\begin{equation}
\begin{aligned}
\label{eq:pre_off}
\off(\hat{\Sigma}^{-1}\hat{Y}\herm\hat{Y}\hat{\Sigma}^{-1})
\leq{}& \frac{1}{1-\epsU}\bigl(\lVert \hat{\Sigma}^{-1}\Delta R\rVert_\fro^2+2(1+\epsU)\lVert \hat{\Sigma}^{-1}\Delta R\rVert_\fro\bigr)+2\epsU+\frac{\epsU}{1-\epsU}
+O(\machepslow^2)\\
&+2(1+\epsU)\bigl(1+\epsU+\lVert \hat{\Sigma}^{-1}\Delta R\rVert_\fro\bigr)\lVert E_Y\rVert_\fro\lVert\hat{\Sigma}^{-1}\rVert_2
+\lVert E_Y\rVert_\fro^2\lVert\hat{\Sigma}^{-1}\rVert_2^2.
\end{aligned}
\end{equation}
Then we only need to estimate \(\lVert \hat{\Sigma}^{-1}\Delta R\rVert_\fro\).

Since the diagonal elements of \(\hat{\Sigma}\) are in descending order and
\(\Delta R\) is an upper triangular matrix, we have
\begin{equation}
\label{eq:pre_SigmaDeltaR}
\lVert \hat{\Sigma}^{-1}\Delta R\rVert_\fro
\leq \lVert \hat{\Sigma}^{-1}\Delta R\hat{\Sigma}\hat{\Sigma}^{-1}\rVert_\fro
\leq \lVert \Delta R\hat{\Sigma}^{-1}\rVert_\fro.
\end{equation}
The relationship between \(\Delta RR^{-1}\) and \(\Delta R\hat{\Sigma}^{-1}\) is
\begin{equation*}
\Delta RR^{-1}
=\Delta R\bigl(I+\up(E_U)\bigr)^{-1}\hat{\Sigma}^{-1}
=\Delta R\hat{\Sigma}^{-1}+\Delta R\bigl(\bigl(I+\up(E_U)\bigr)^{-1}-I\bigr)\hat{\Sigma}^{-1}.
\end{equation*}
Combined with~\eqref{eq:pre_SigmaDeltaR}, it follows that
\begin{equation}
\lVert \hat{\Sigma}^{-1}\Delta R\rVert_\fro
\leq \lVert \Delta RR^{-1}\rVert_\fro+\frac{\epsU}{1-\epsU}\lVert\Delta R\rVert_\fro\lVert \hat{\Sigma}^{-1}\rVert_2
\end{equation}
Note that~\eqref{eq:pre_QR3} and the assumption~\eqref{eq:pre_assump} allow us
to apply Lemma~\ref{lemma:perturQR} to bound \(\lVert \Delta R\rVert_\fro\)
and \(\lVert \Delta RR^{-1}\rVert_\fro\), i.e.,
\begin{align*}
\lVert \Delta R\rVert_\fro
&\leq 8\frac{1+\epsU}{1-\epsU}\lVert \hat{\Sigma}\low(E_U)+\Tilde{V}\herm\Delta X\herm\hat{U}+\Tilde{V}\herm E_1\rVert_\fro\\
&\leq 8\frac{1+\epsU}{1-\epsU}\bigl(\lVert \hat{\Sigma}\low(E_U)\rVert_\fro+\lVert\Delta X\herm\hat{U}\rVert_\fro+\lVert E_1\rVert_\fro\bigr)\\
&\leq 8\frac{1+\epsU}{1-\epsU}\bigl(\epsU\lVert\hat{\Sigma}\rVert_2+\sqrt{n}(1+\epsU)\bigl(\epsSVD+\sqrt{n}\gamma_n+\epsqr+\sqrt{n}\gamma_n\epsqr\bigr)\lVert X\rVert_2\bigr)
\end{align*}
and
\begin{align*}
\lVert \Delta RR^{-1}\rVert_\fro
&\leq (\sqrt{2}+\sqrt{3}\,)\lVert \bigl(\hat{\Sigma}\low(E_U)+\Tilde{V}\herm\Delta X\herm\hat{U}+\Tilde{V}\herm E_1\bigr)\hat{\Sigma}^{-1}\rVert_\fro\\
&\leq (\sqrt{2}+\sqrt{3}\,)\bigl(\lVert \hat{\Sigma}\low(E_U)\hat{\Sigma}^{-1}\rVert_\fro+\lVert\Delta X\herm\hat{U}\hat{\Sigma}^{-1}\rVert_\fro+\lVert E_1\hat{\Sigma}^{-1}\rVert_\fro\bigr)\\
&\leq (\sqrt{2}+\sqrt{3}\,)\bigl(\epsU+\lVert\Delta X\herm\hat{U}\hat{\Sigma}^{-1}\rVert_\fro+O(\macheps)\lVert X\rVert_2\lVert\hat{\Sigma}^{-1}\rVert_2\bigr).
\end{align*}
It remains to estimate
\(\lVert\Delta X\herm\hat{U}\hat{\Sigma}^{-1}\rVert_\fro\).
By~\eqref{eq:pre_X}, we obtain
\begin{align*}
\lVert\Delta X\herm\hat{U}\hat{\Sigma}^{-1}\rVert_\fro
&\leq \sqrt{n}\,\epsSVD(1+\epsU)\lVert X\rVert_2\lVert\hat{\Sigma}^{-1}\rVert_2.
\end{align*}
Combined with~\eqref{eq:pre_off}, we conclude~\eqref{eq:pre-weak}.

Further, if~\eqref{eq:pre_XJacobi} holds, we can obtain
\begin{align*}
\lVert\Delta X\herm\hat{U}\hat{\Sigma}^{-1}\rVert_\fro
&=\lVert\hat{\Sigma}^{-1}\hat{U}\herm\Delta X\rVert_\fro\\
&=\lVert\hat{\Sigma}^{-1}\bigl(\hat{U}^{-1}\hat{U}\bigr)\hat{U}\herm\Delta X\rVert_\fro\\
&\leq \lVert\Tilde{V}\hat{\Sigma}^{-1}\hat{U}^{-1}\Delta X\rVert_\fro
+\lVert \hat{\Sigma}^{-1}\hat{U}^{-1}E_U\Delta X\rVert_\fro\\
&\leq \lVert \bigl(I-X^{-1}\Delta X\bigr)X^{-1}\Delta X\rVert_\fro
+O(\machepslow\macheps)\lVert X\rVert_2\lVert\hat{\Sigma}^{-1}\rVert_2
\end{align*}
Together with
\[
\lVert X^{-1}\Delta X\rVert_\fro\leq\sqrt{n}\,\epsSVD\lVert Z_r^{-1}\rVert_2,
\]
we have
\[
\lVert\Delta X\herm\hat{U}\hat{\Sigma}^{-1}\rVert_\fro
\leq \sqrt{n}\,\epsSVD\lVert Z_r^{-1}\rVert_2+O(\machepslow^2+\macheps)\lVert X\rVert_2\lVert\hat{\Sigma}^{-1}\rVert_2.
\]
Combined with~\eqref{eq:pre_off}, we arrive at~\eqref{eq:pre-strong} because
\begin{equation}
\label{eq:norm_sigmainv_deltaR}
\lVert \hat{\Sigma}^{-1}\Delta R\rVert_\fro
\leq(\sqrt{2}+\sqrt{3}\,)\bigl(\epsU+\sqrt{n}\,\epsSVD\lVert Z_r^{-1}\rVert_2\bigr)+O(\machepslow^2+\macheps)\lVert X\rVert_2\lVert\hat{\Sigma}^{-1}\rVert_2.
\end{equation}
\end{proof}

The effectiveness of our mixed precision Jacobi SVD algorithm is largely based
on the quadratic convergence of the Jacobi algorithm.
Our mixed precision Jacobi SVD algorithm also inherits this property under
mild assumptions.
In order to show this, we first derive the following theorem.

\begin{theorem}
\label{thm:pre_simple}
Assume that \eqref{eq:pre_X}--\eqref{eq:pre_Y} hold together with
\begin{equation}
\label{eq:pre_assump_simple}
\epsU\lVert\hat{\Sigma}\rVert_2\lVert\hat{\Sigma}^{-1}\rVert_2
+\sqrt{n}\,\bigl(\epsSVD+\sqrt{n}\,\gamma_n+\epsqr+\sqrt{n}\,\gamma_n\epsqr\bigr)(1+\epsU)
\lVert X\rVert_2\lVert\hat{\Sigma}^{-1}\rVert_2
<\sqrt{\frac32}-1.
\end{equation}
Then the matrix \(\hat{Y}\) obtained by~\eqref{eq:pre_Y} satisfies
\[
\off(\hat{Y}\herm\hat{Y})
:=\biggl(\sum_{i=1}^n\sum_{j=1, j\neq i}^n \lvert (\hat{Y}\herm\hat{Y})_{ij}
\rvert^2\biggr)^{1/2}
\leq\delta,
\]
where
\begin{equation}
\label{eq:delta}
\delta=\frac{17\epsU}{1-\epsU}\lVert\hat{\Sigma}\rVert_2^2
+\frac{16(1+\epsU)}{1-\epsU}\sqrt{n}\,\epsSVD\lVert X\rVert_2\lVert\hat{\Sigma}\rVert_2
+O(\macheps+\machepslow^2)\lVert X\rVert_2\lVert\hat{\Sigma}\rVert_2.
\end{equation}
\end{theorem}

\begin{proof}
By~\eqref{eq:pre_finalY} and~\eqref{eq:pre_QR2}, we have
\begin{align}
\off(\hat{Y}\herm\hat{Y})
\leq{}& \off\bigl(\hat{R}\hat{U}^{-1}\hat{U}\iherm\hat{R}\herm\bigr)
+2(1+\epsU)\lVert \hat{R}\rVert_2\lVert E_Y\rVert_\fro+
\lVert E_Y\rVert_\fro^2\nonumber\\
\leq{}& \off\bigl((\Delta R+\hat{\Sigma})(I+E_U)^{-1}(\Delta R\herm+\hat{\Sigma})\bigr)
+2(1+\epsU)\lVert \hat{R}\rVert_2\lVert E_Y\rVert_\fro+
\lVert E_Y\rVert_\fro^2\nonumber\\
\leq{}& \frac{1}{1-\epsU}\bigl(\lVert \Delta R\rVert_\fro^2
+2\lVert \Delta R\rVert_\fro\lVert\hat{\Sigma}\rVert_2\bigr)
+\frac{\epsU}{1-\epsU}\lVert\hat{\Sigma}\rVert_2^2\nonumber\\
&+2(1+\epsU)\bigl(\lVert \hat{\Sigma}\rVert_2+\lVert \Delta R\rVert_\fro\bigr)
\lVert E_Y\rVert_\fro+\lVert E_Y\rVert_\fro^2.
\label{eq:pre_offY}
\end{align}
Then we only need to estimate \(\lVert \Delta R\rVert_\fro\).
By~\eqref{eq:pre_E1}, it follows that
\begin{align*}
\lVert \Delta R\rVert_\fro
&\leq 8\lVert F_1\rVert_\fro\\
&\leq 8\bigl(\epsU\lVert \hat{\Sigma}\rVert_2+(1+\epsU)\epsSVD\lVert X\rVert_2+\lVert E_1\rVert_\fro\bigr)\\
&\leq 8\epsU\lVert \hat{\Sigma}\rVert_2+8(1+\epsU)\epsSVD\lVert X\rVert_2+O(\macheps)\lVert X\rVert_2.
\end{align*}
Substituting this estimate into~\eqref{eq:pre_offY} yields the conclusion.
\end{proof}

Using Theorem~\ref{thm:pre_simple} and the theory developed in~\cite{VK1966},
we obtain Theorem~\ref{thm:quadratic-convergence} on quadratic convergence.
Although rounding error is not taken into account,
Theorem~\ref{thm:quadratic-convergence} still provides an explanation for the
rapid convergence of our mixed precision Jacobi SVD algorithm after switch
from lower precision back to working precision.

\begin{theorem}
\label{thm:quadratic-convergence}
Let \(Y^{(k)}\) be the matrix \(\hat Y\) after \(k\) Jacobi sweeps.
If
\[
4\sqrt{2}\,\delta< d\leq \min_{\sigma_i(Y^{(k)})\neq \sigma_j(Y^{(k)})}
\bigl\lvert \sigma_i^2(Y^{(k)})-\sigma_j^2(Y^{(k)})\bigr\rvert,
\]
where \(\delta\) is defined in~\eqref{eq:delta}, then
\[
\off\bigl((Y^{(k+1)}\bigr)\herm Y^{(k+1)})
\leq \frac{\sqrt{34}}{3d}\bigl(\off\bigl((Y^{(k)})\herm Y^{(k)}\bigr)\bigr)^2.
\]
\end{theorem}

\section{Efficiency of the method to switch precisions}
\label{sec:analysis_efficiency_switch}
In the discussion of Section~\ref{subsec:switch}, we have mentioned two
methods to switch precisions:
one is
\[
B \gets X\herm U,\quad Q_1R_1 \gets B,\quad Y_1 \gets XQ_1;
\]
the other is
\[
Q_2R_2 \gets V,\quad Y_2 \gets XQ_2,
\]
where \(U\) and \(V\) satisfy \(X = U\Sigma V\herm\).
Although mathematically we have \(Q_1=Q_2\) and \(Y_1=Y_2\) in exact arithmetic,
the computed results are different if rounding errors are taken
into account.
We use subscripts to represent computed results by different methods.

In practice the second method is more costly because \(V\) needs to be
accumulated.
Hence the first method is preferred from a performance viewpoint.
In the following we show that \(Y_1\) is better than~\(Y_2\) in finite
precision arithmetic, in the sense that the columns of \(Y_1\) are more
orthogonal.

Let us regard the result computed in working precision as the accurate one.
Suppose the computed quantities computed by these two methods satisfy
\begin{align}
\label{eq:Y1}
X\herm\hat{U}=Q_1R_1,\quad Y_1=XQ_1,
\end{align}
and
\begin{align}
\label{eq:Y2}
Y_2=X\Tilde{V},
\end{align}
respectively, where \(\hat{U}\) and \(\Tilde{V}\) satisfy~\eqref{eq:pre_XJacobi}.
Theorem~\ref{thm:effciency_switch} characterizes the orthogonality of \(Y_1\)
and \(Y_2\) after normalizing the columns.

\begin{theorem}
\label{thm:effciency_switch}
Under the assumptions that~\eqref{eq:pre_XJacobi}, \eqref{eq:pre_assump},
\eqref{eq:Y1} and~\eqref{eq:Y2}, it holds that
\[
\lVert Y_1\hat{\Sigma}^{-1}-\hat{U}\rVert_2\leq\epsU+(1+\epsU)\bigl(\epsU+(\sqrt{2}+\sqrt{3}\,)\bigl(\epsU+\sqrt{n}\,\epsSVD\lVert Z_r^{-1}\rVert_2\bigr)\bigr)
\]
and
\[
\lVert Y_2\hat{\Sigma}^{-1}-\hat{U}\rVert_2\leq n\epsSVD\lVert X\rVert_2\lVert\hat{\Sigma}^{-1}\rVert_2.
\]
\end{theorem}

\begin{proof}
First we analyze \(Y_1\) computed by the first method.
By Lemma~\ref{lemma:hatY} and the assumptions, we easily obtain
\begin{align*}
Y_1 &= \hat{U}\iherm \hat{R}\herm
=\hat{U}\iherm \bigl(I+\up(E_U)\herm\bigr)\hat{\Sigma}+\hat{U}\iherm\Delta R\herm
=\hat{U}\iherm\bigl(I+\up(E_U)\herm+\Delta R\herm\hat{\Sigma}^{-1}\bigr)\hat{\Sigma}
\end{align*}
with \(\Delta R=\hat{R}-\hat{\Sigma}\bigl(I+\up(E_U)\bigr)\).
Similarly to~\eqref{eq:norm_sigmainv_deltaR}, we draw the first conclusion
because
\begin{align*}
\lVert Y_1(:, i)\rVert_2
&\leq \lVert \hat{U}\iherm\rVert_2\bigl\lVert I+\up(E_U)\herm+\Delta R\herm\hat{\Sigma}^{-1}\bigr\rVert_2\lVert \hat{\Sigma}(:, i)\rVert_2\\
&\leq (1+\epsU)\bigl(1+\epsU+(\sqrt{2}+\sqrt{3}\,)\bigl(\epsU+\sqrt{n}\,\epsSVD\lVert Z_r^{-1}\rVert_2\bigr)\bigr)\lVert\hat{\Sigma}(:, i)\rVert_2.
\end{align*}

Considering \(Y_2\) computed by the second method, we have
\[
Y_2=X\Tilde{V}=\hat{U}\hat{\Sigma}+\Delta X\Tilde{V}
\]
and then
\[
Y_2\hat{\Sigma}^{-1}=\hat{U}+\Delta X\Tilde{V}\hat{\Sigma}^{-1}.
\]
Together with
\[
\lVert \Delta X\Tilde{V}\hat{\Sigma}^{-1}\rVert_2
\leq n\epsSVD\lVert X\rVert_2\lVert\hat{\Sigma}^{-1}\rVert_2,
\]
we conclude the proof.
\end{proof}

Theorem~\ref{thm:effciency_switch} shows that both \(Y_1\hat{\Sigma}^{-1}\)
and \(Y_2\hat{\Sigma}^{-1}\) are close to have orthogonal columns.
However, because \(\lVert Z_r^{-1}\rVert_2\) is often much smaller than
\(\lVert X\rVert_2\lVert\hat{\Sigma}^{-1}\rVert_2\), the columns of \(Y_1\)
are thus more orthogonal compared to those of \(Y_2\).
Therefore, the first method is preferred both computationally and numerically.

\section{Numerical Experiments}
\label{sec:experiments}
In this section we present numerical experiments on the x86-64 architecture of
our mixed precision algorithm described in Algorithm~\ref{alg:msvj}.
We choose IEEE double precision as the working precision, and IEEE
single precision as the lower precision.
Our implementation is derived from \texttt{XGEJSV} and \texttt{XGESVJ} in
LAPACK version 3.9.0, and makes use of \texttt{XGEQRF}, \texttt{XGEQP3},
\texttt{xGESVJ}, and \texttt{xGESVD} for the computations in single precision,
where \(\mathtt{X}\in\set{\mathtt{D},\mathtt{Z}}\) and
\(\mathtt{x}\in\set{\mathtt{S},\mathtt{C}}\).
All experiments have been performed on a Linux cluster with two twelve-core
Intel Xeon E5-2670 v3 2.30~GHz CPUs and 128~GB of main memory.
Programs are compiled using the GNU Fortran compiler version 6.3.1 with
optimization flag \texttt{-O2}, and are linked with the OpenBLAS library
version~0.3.9.

\subsection{Experiment settings}
In the following, we compare Algorithm~\ref{alg:msvj} with the double
precision solver \texttt{XGEJSV} in LAPACK.
\texttt{XGEJSV} is a good implementation of Algorithm~\ref{alg:presvj},
which mainly consists of two stages: QR preconditioning and the
one-sided Jacobi SVD algorithm (i.e., Algorithm~\ref{alg:svj}).
In LAPACK, Algorithm~\ref{alg:svj} is implemented in \texttt{XGESVJ}.

All left and right singular vectors are explicitly requested in these solvers.
We report the relative run time, which is the ratio of the wall clock time of
a solver (or its components) over that consumed by \texttt{XGEJSV}.
There are several tunable parameters required by Algorithm~\ref{alg:msvj}.
In the following experiments, we choose these parameters as follows:
\(\tolalg=10^{-2}\), \(\tolorth=10^{-5}\), \(\tolcond=1.5\cdot n^{1/4}\).

In performance plots, the labels `QR pre', `xGESVD\slash{}xGESVJ',
`switch' and `XGESVJ', respectively, represent the four components of
Algorithm~\ref{alg:msvj}: QR preconditioning, the single precision SVD
algorithm as a preconditioning, switching precisions, and the double
precision one-sided Jacobi SVD algorithm as a refinement.
Moreover, the numbers in the bar mean the number of sweeps performed by the
one-sided Jacobi SVD algorithm (in the working precision).

Similar to~\cite{DV2008b}, the test matrices are randomly generated matrices
of the form \(A=BD\), where $B$ is with unit column norms and $D$ is a
diagonal matrix.
We can prescribe the condition numbers of \(D\) and \(B\) as follows.
We first use \(\texttt{XLATM1}\) to generate two diagonal matrices \(D\)
and~\(\Sigma\) with given condition numbers \(\kappa(D)\) and
\(\kappa(\Sigma)\), respectively.
Then the matrix \(B\) is constructed through \(B=W_1\Sigma W_2W_3\), where
\(W_1\) and \(W_2\) are unitary factors in the QR factorization of randomly
generated matrices, and \(W_3\) is carefully chosen so that \(B\) has unit
column norms~\cite{CL1983,DH2000}.

In our experiments, we test \(16\) combinations of \(\kappa(D)\) and
\(\kappa(B)\) as shown in Table~\ref{tab:testmat}.
The parameter \texttt{MODE} is used by \texttt{XLATM1} to specify different
types of test matrices; see Table~\ref{tab:XLATM1} for details.
Both square and rectangular matrices are tested.
For square matrices, we choose two typical cases:
\(\bigl(\kappa(D),\kappa(B)\bigr)=(10^2,10^{12})\) and
\(\bigl(\kappa(D),\kappa(B)\bigr)=(10^{20},10^2)\);
for rectangular matrices, we only test the case
\(\bigl(\kappa(D),\kappa(B)\bigr)=(10^2,10^{12})\).

\begin{table}[!tb]
\centering
\caption{The description of different types of test matrices.}
\label{tab:testmat}
\begin{tabular}{|c|c|c|c|c|c|c|c|c|c|c|c|c|c|c|c|c|}
\hline
ID of matrices& 1& 2& 3& 4& 5& 6& 7& 8& 9& 10& 11&
12& 13& 14& 15& 16\\
\hline
\texttt{MODE} of \(D\)& 1& 1& 1& 1& 2& 2& 2& 3& 3& 3& 4&
4& 4& 5& 5& 5\\
\hline
\texttt{MODE} of \(\Sigma\)& 2& 3& 4& 5& 3& 4& 5& 2& 4& 5& 2&
3& 5& 2& 3& 4\\
\hline
\end{tabular}
\end{table}

\begin{table}[!tb]
\centering
\caption{Explanation of different modes of \texttt{XLATM1} for generating
\(D=\diag(d)\).}
\label{tab:XLATM1}
\begin{tabular}{cp{0.8\textwidth}}
\hline
\texttt{MODE} & \hfil Description\hfil \\
\hline
1 & clustered at $1/\kappa$, where $d(1) = 1$, $d(2:n) = 1/\kappa$ \\
2 & clustered at $1$, where $d(1:n-1) = 1$, $d(n) = 1/\kappa$ \\
3 & Geometric distribution, where $d(i) = 1/\kappa^{(i-1)/(n-1)}$ \\
4 & Arithmetic distribution, where $d(i) = 1 - (i-1)/(n-1)(1 - 1/\kappa)$ \\
5 & Log-random distribution, where $d(i)\in(1/\kappa,1)$ so that their logarithms are random and uniformly
distributed \\
\hline
\end{tabular}
\end{table}

\subsection{Performance tests of the basic mixed precision Jacobi SVD algorithm}
\label{subsec:experiments-simplemsvj}
We first show the performance of the basic mixed precision algorithm
(Algorithm~\ref{alg:basic-msvj}) in Figure~\ref{fig:msvj_simple}.
It can be seen that Algorithm~\ref{alg:basic-msvj} only can achieve about
\(1.1\times\)--\(1.4\times\) speedup compared to \texttt{DGEJSV}.
This is far below the ideal speedup on the x86-64 architecture, which is
\(2\times\).
In fact, we can see from Figure~\ref{fig:msvj_simple} that the computation in
single precision does not effectively accelerate the convergence in double
precision because the refinement stage (Line~\ref{line:basic-msvj-refinement})
still requires too many Jacobi iterations to converge.
Therefore, Algorithm~\ref{alg:basic-msvj} is not very satisfactory.

\begin{figure}[!tb]
\centering
\includegraphics[width=\textwidth]{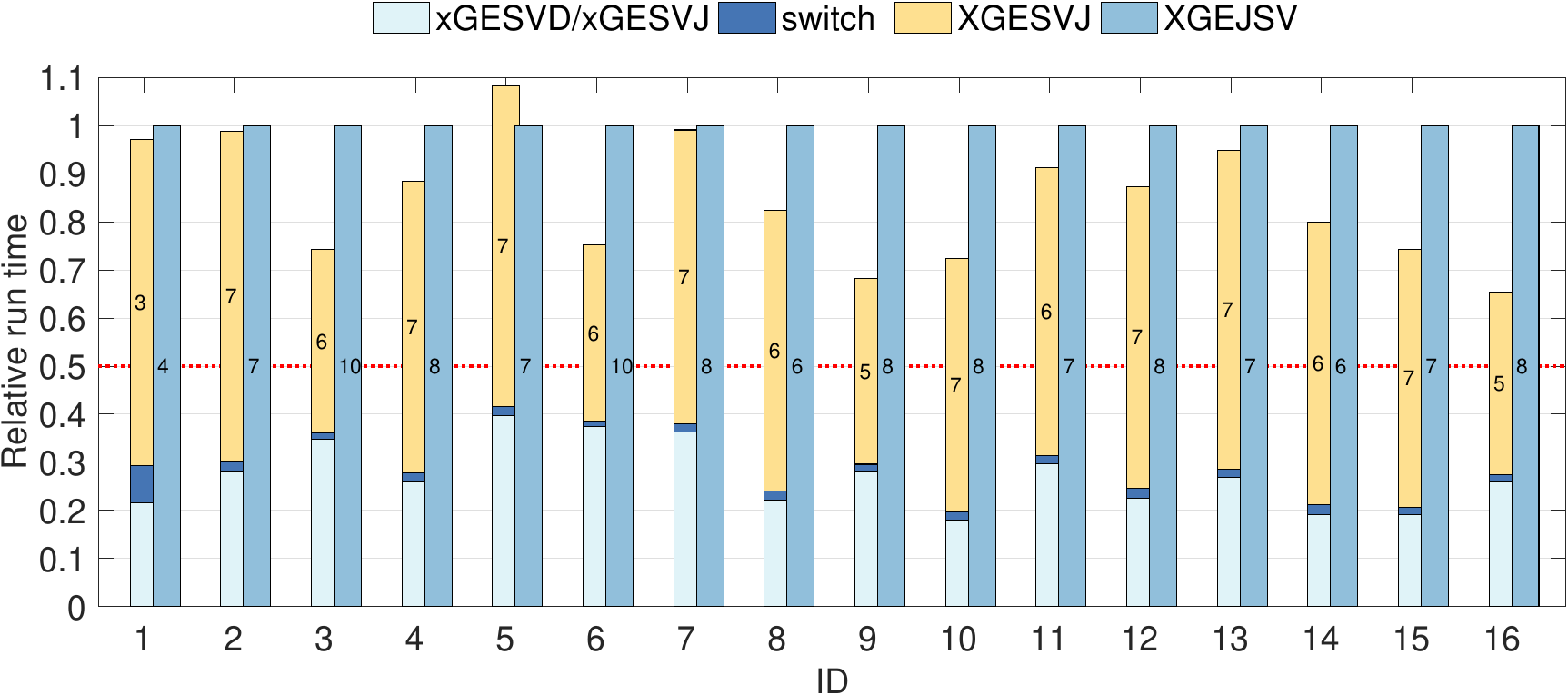}
\caption{Relative run time of Algorithm~\ref{alg:basic-msvj} for
$4096\times4096$ real matrices with
\(\bigl(\kappa(D),\kappa(B)\bigr)=(10^2,10^{12})\).
The numbers in the bars denote the numbers of iterations required by the Jacobi algorithm.}
\label{fig:msvj_simple}
\end{figure}

\subsection{Performance tests of the mixed precision RRQR algorithm}
In this subsection, we show the relative run time of Algorithm~\ref{alg:mprrqr} in
Figures~\ref{fig:mrrqr_4096_1e21e12} and~\ref{fig:mrrqr_4096_1e201e2}.
It can be seen that the speedup of Algorithm~\ref{alg:mprrqr} over
LAPACK's \texttt{DGEQP3} depends on the test matrices.
For most matrices it achieves about \(2\times\) speedup in general.
Therefore, it is worth developing a mixed precision algorithm for RRQR.

\begin{figure}[!tb]
\centering
\includegraphics[width=\textwidth]{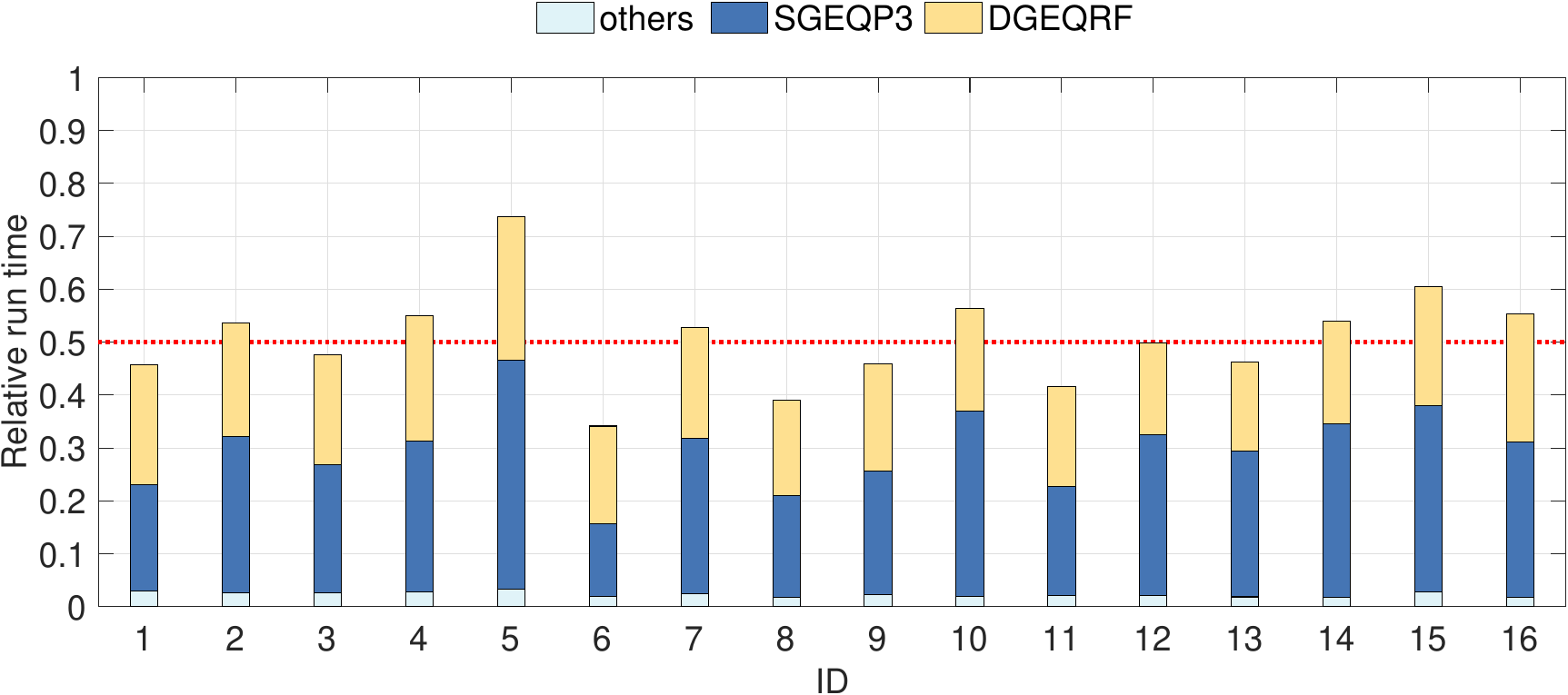}
\caption{Relative run time of Algorithm~\ref{alg:mprrqr} for
$4096\times4096$ real matrices with
\(\bigl(\kappa(D),\kappa(B)\bigr)=(10^2,10^{12})\).}
\label{fig:mrrqr_4096_1e21e12}
\end{figure}

\begin{figure}[!tb]
\centering
\includegraphics[width=\textwidth]{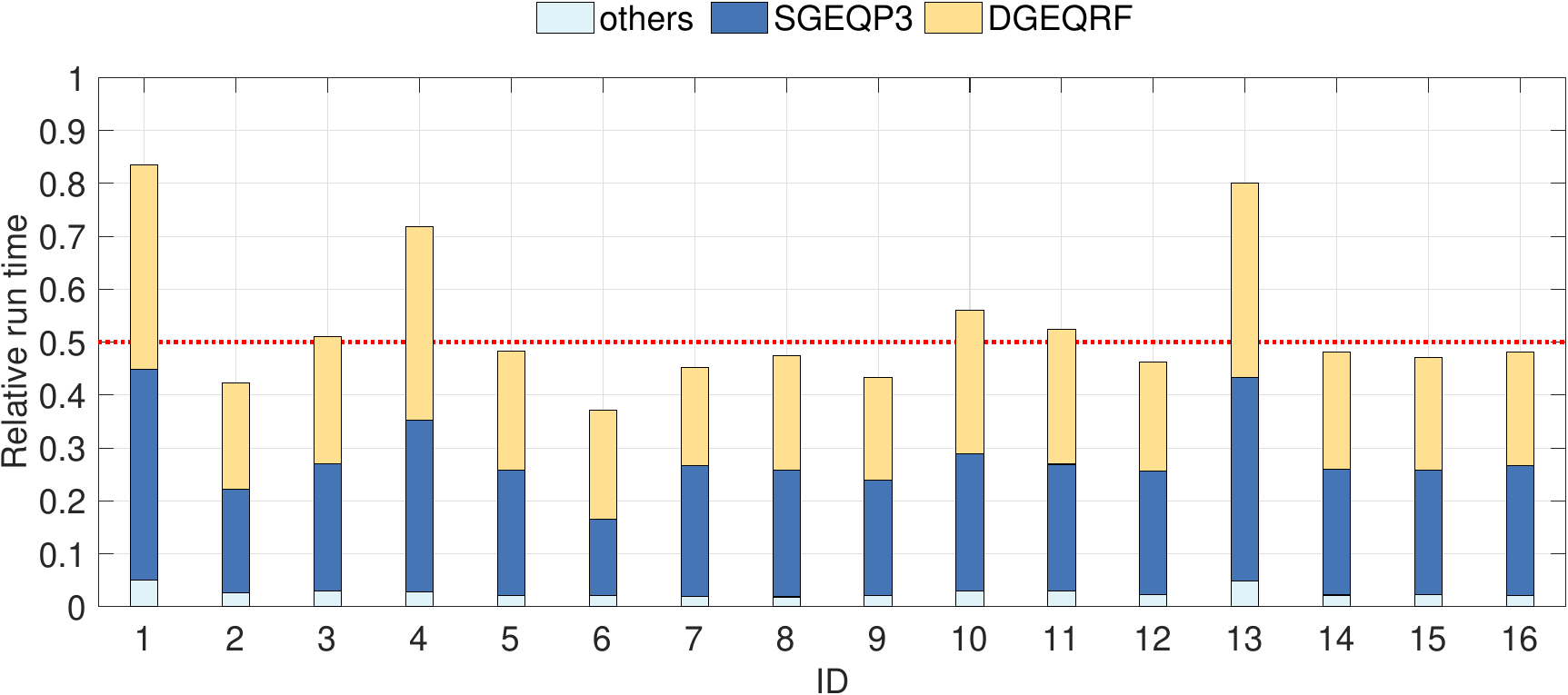}
\caption{Relative run time of Algorithm~\ref{alg:mprrqr} for
$4096\times4096$ real matrices with
\(\bigl(\kappa(D),\kappa(B)\bigr)=(10^{20},10^{2})\).}
\label{fig:mrrqr_4096_1e201e2}
\end{figure}

\subsection{Performance tests of the mixed precision Jacobi SVD algorithm}
Figures~\ref{fig:double_1e21e12_square_fixedsize}
and~\ref{fig:double_1e201e2_square_fixedsize} show the relative run time of
the SVD for \(4096\times4096\) real matrices computed by Algorithm~\ref{alg:msvj}.
It can be seen that when $\kappa(D)$ is moderate, Algorithm~\ref{alg:msvj}
achieves about \(2\times\) speedup in general compared to \texttt{DGEJSV}.
There are several cases (\(\mathrm{ID}\in\set{8,11,14}\)) that the speedup is
only about~\(20\%\).
These matrices are all generated with \(\mathtt{MODE}(\Sigma)=2\).
In these cases, the QR preconditioning stage is very effective so
that after preconditioning the columns of \(X\) are close to orthogonal with
\(X(:,i)\herm X(:,j)\approx10^{-2}\) for \(i\neq j\).
However, it still takes a lot of effort for the single precision Jacobi SVD
algorithm to improve the orthogonality from \(O(10^{-2})\) to \(O(10^{-6})\).
Thus, the performance gain for using a mixed precision algorithm is limited
for these cases.
When $\kappa(D)$ is extremely large, there are more cases that the run time of
Algorithm~\ref{alg:msvj} is almost the same as that of \texttt{DGEJSV} as
shown in Figure~\ref{fig:double_1e201e2_square_fixedsize}.
In these cases, the QR preconditioning stage is sufficiently
effective, so that the one-sided Jacobi algorithm in double
precision can converge very rapidly after preconditioning, making the SVD in
single precision useless.
By the techniques mentioned in Section~\ref{subsec:special-cases}, we can
detect these cases in an early stage and directly move on to the double precision
Jacobi SVD algorithm.
Algorithm~\ref{alg:msvj} is still slightly faster than \texttt{DGEJSV},
because the time spent on the RRQR algorithm is reduced by exploiting mixed
precision techniques.

\begin{figure}[!tb]
\centering
\includegraphics[width=\textwidth]{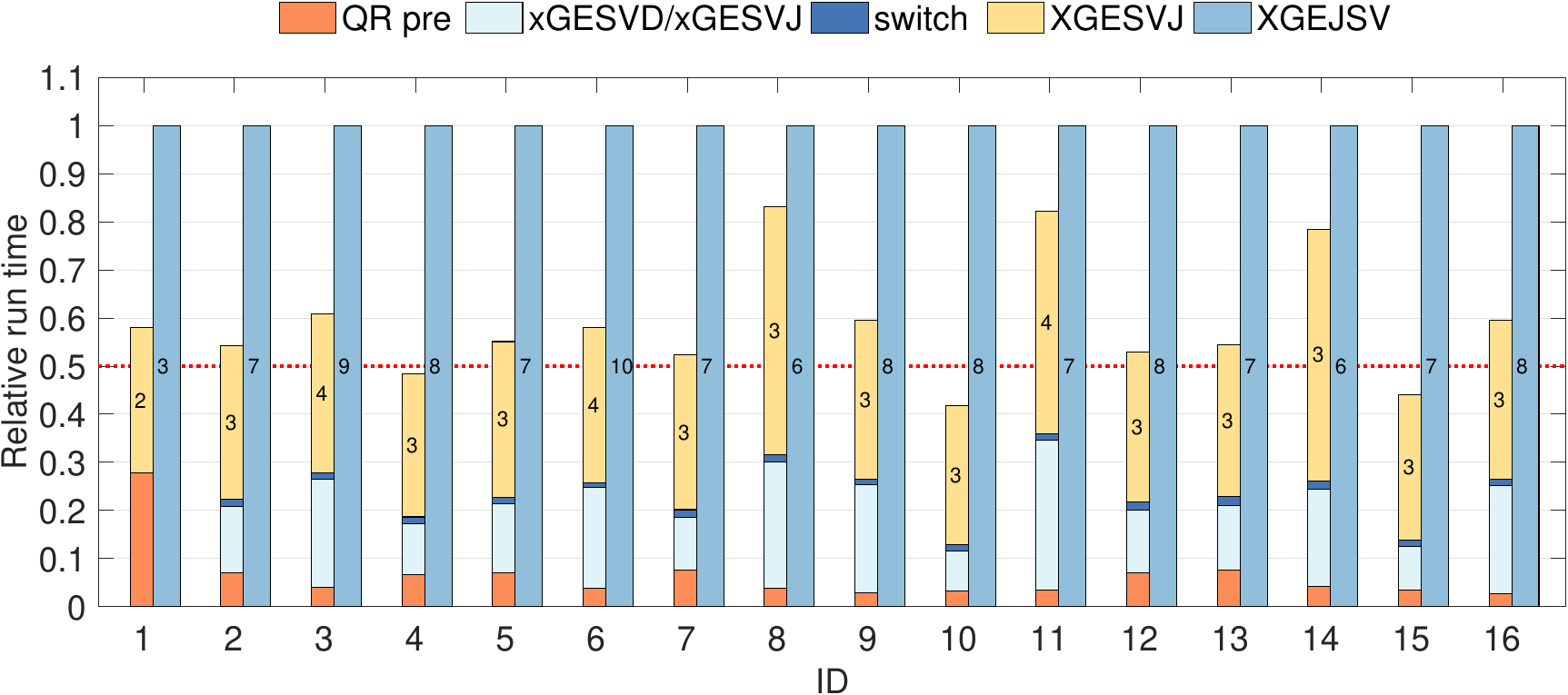}
\caption{Relative run time of Algorithm~\ref{alg:msvj} for
$4096\times4096$ real matrices with
\(\bigl(\kappa(D),\kappa(B)\bigr)=(10^2,10^{12})\).
The numbers in the bars denote the numbers of iterations required by the Jacobi algorithm.}
\label{fig:double_1e21e12_square_fixedsize}
\end{figure}

\begin{figure}[!tb]
\centering
\includegraphics[width=\textwidth]{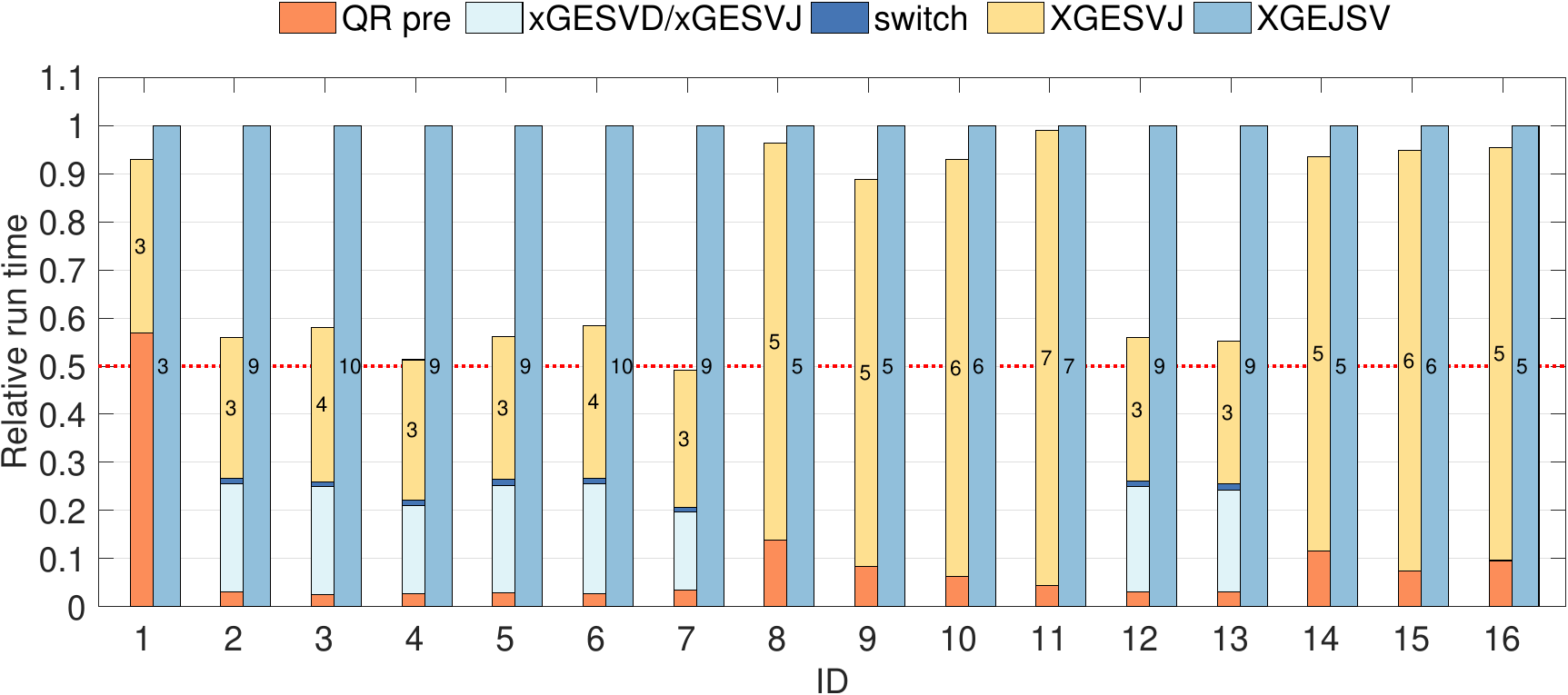}
\caption{Relative run time of Algorithm~\ref{alg:msvj} for
$4096\times4096$ real matrices with
\(\bigl(\kappa(D),\kappa(B)\bigr)=(10^{20},10^2)\).
The numbers in the bars denote the numbers of iterations required by the Jacobi algorithm.}
\label{fig:double_1e201e2_square_fixedsize}
\end{figure}

Similar results for \(4096\times4096\) complex matrices are shown in
Figures~\ref{fig:complex_1e21e12_square_fixedsize}
and~\ref{fig:complex_1e201e2_square_fixedsize}.
The speedup of Algorithm~\ref{alg:msvj} over \texttt{ZGEJSV} is about
\(2\times\) for most matrices, while those generated by
\(\mathtt{MODE}(\Sigma)=2\) have much lower speedup.
For a few test matrices the speedup is even higher than \(3\times\),
mainly because \texttt{ZGEJSV} ran into a conservative branch that requires
accumulating \(V\) explicitly.
Algorithm~\ref{alg:msvj} safely avoids the accumulation of \(V\) in these
cases.

\begin{figure}[!tb]
\centering
\includegraphics[width=\textwidth]{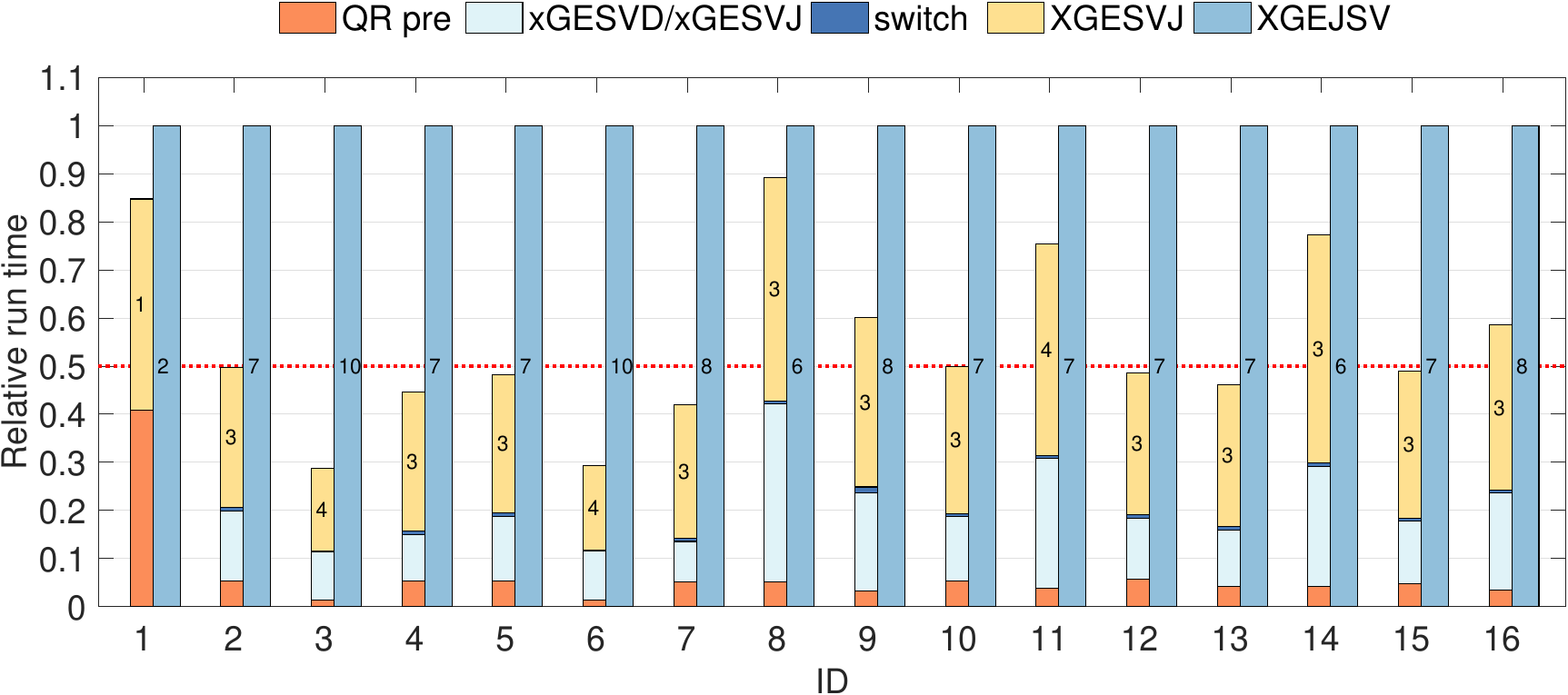}
\caption{Relative run time of Algorithm~\ref{alg:msvj} for
$4096\times4096$ complex matrices with
\(\bigl(\kappa(D),\kappa(B)\bigr)=(10^2,10^{12})\).
The numbers in the bars denote the numbers of iterations required by the Jacobi algorithm.}
\label{fig:complex_1e21e12_square_fixedsize}
\end{figure}

\begin{figure}[!tb]
\centering
\includegraphics[width=\textwidth]{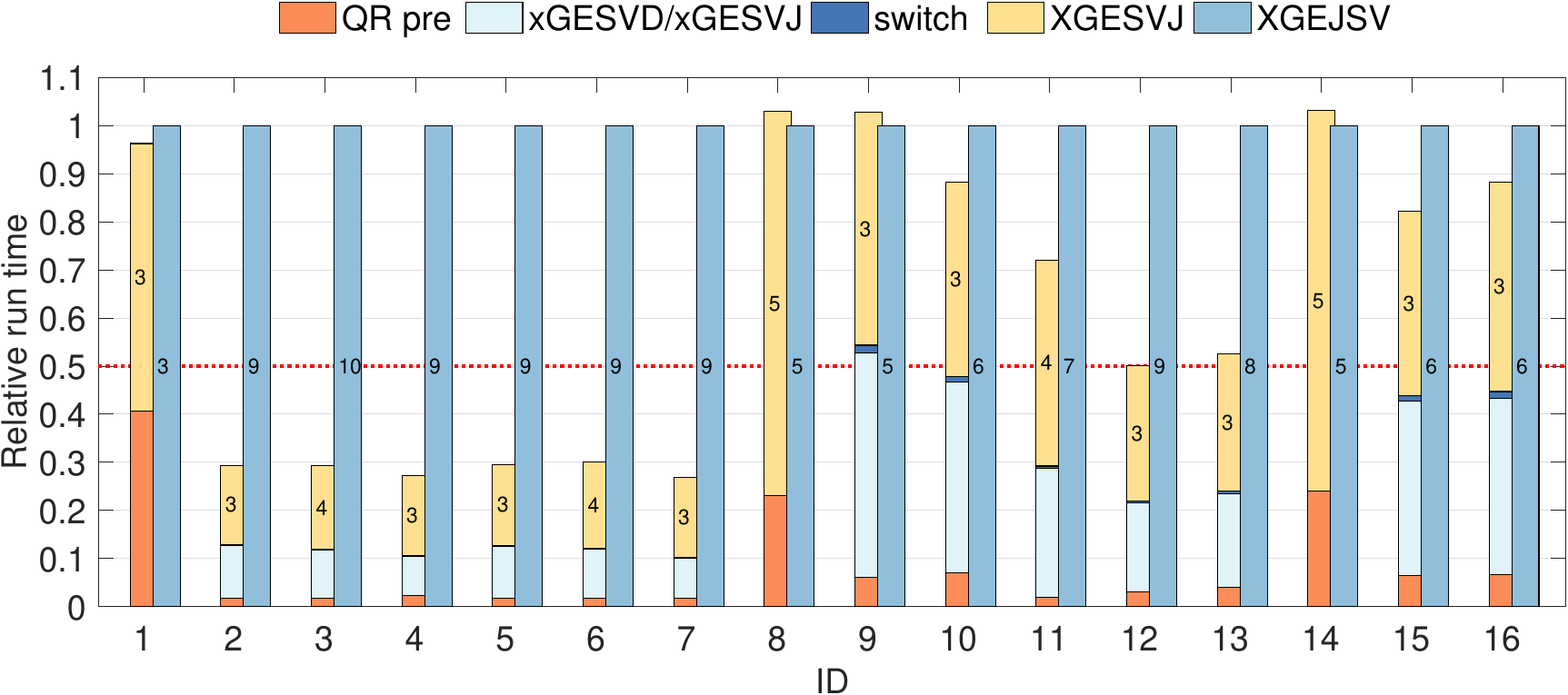}
\caption{Relative run time of Algorithm~\ref{alg:msvj} for
$4096\times4096$ complex matrices with
\(\bigl(\kappa(D),\kappa(B)\bigr)=(10^{20},10^2)\).
The numbers in the bars denote the numbers of iterations required by the Jacobi algorithm.}
\label{fig:complex_1e201e2_square_fixedsize}
\end{figure}

Figure~\ref{fig:double_1e21e12_rectangle_fixedsize} shows the relative
run time for rectangular matrices of size \(12288\times4096\) with
\(\bigl(\kappa(D),\kappa(B)\bigr)=(10^2,10^{12})\).
The behavior is very similar to that for square matrices.
A \(2\times\) speedup is observed for most matrices, except for those
generated by \(\mathtt{MODE}(\Sigma)=2\).

\begin{figure}[!tb]
\centering
\includegraphics[width=\textwidth]{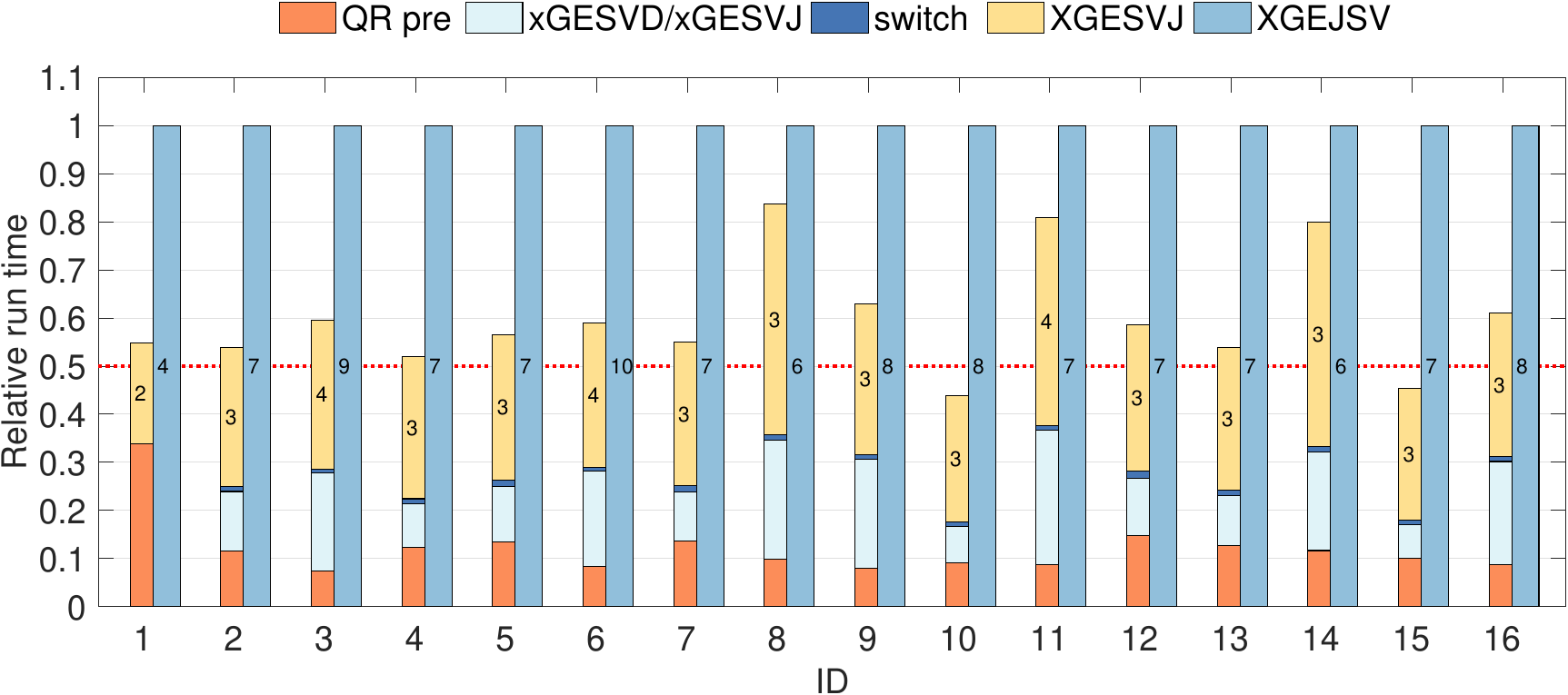}
\caption{Relative run time of Algorithm~\ref{alg:msvj} with
\(\bigl(\kappa(D),\kappa(B)\bigr)=(10^2,10^{12})\) for $12288\times4096$ real
matrices.
The numbers in the bars denote the numbers of iterations required by the Jacobi algorithm.}
\label{fig:double_1e21e12_rectangle_fixedsize}
\end{figure}

In Figure~\ref{fig:dd_1e21e12_fixedsize}, we use double--double%
\footnote{Double--double is a multiple-component format floating-point number,
represented by two \(64\)-bit floating-point numbers.
See, e.g., \cite{HLB2008}.}
and IEEE double precision as our working and lower precisions, respectively, to show
the relative run time of Algorithm~\ref{alg:msvj} for square matrices of size
\(1024\times 1024\) with \(\bigl(\kappa(D),\kappa(B)\bigr)=(10^2,10^{12})\).
In this case, we use the subroutine of MPLAPACK (see~\cite{N2022-MPLAPACK}) to
perform double--double operations.
Since the level 3 BLAS in MPLAPACK is not well-optimized, the proportion of
QR preconditioning and switching precision is very large.
But it still can be seen that our mixed precision algorithm effectively
accelerates the one-sided Jacobi SVD algorithm.

\begin{figure}[!tb]
\centering
\includegraphics[width=\textwidth]{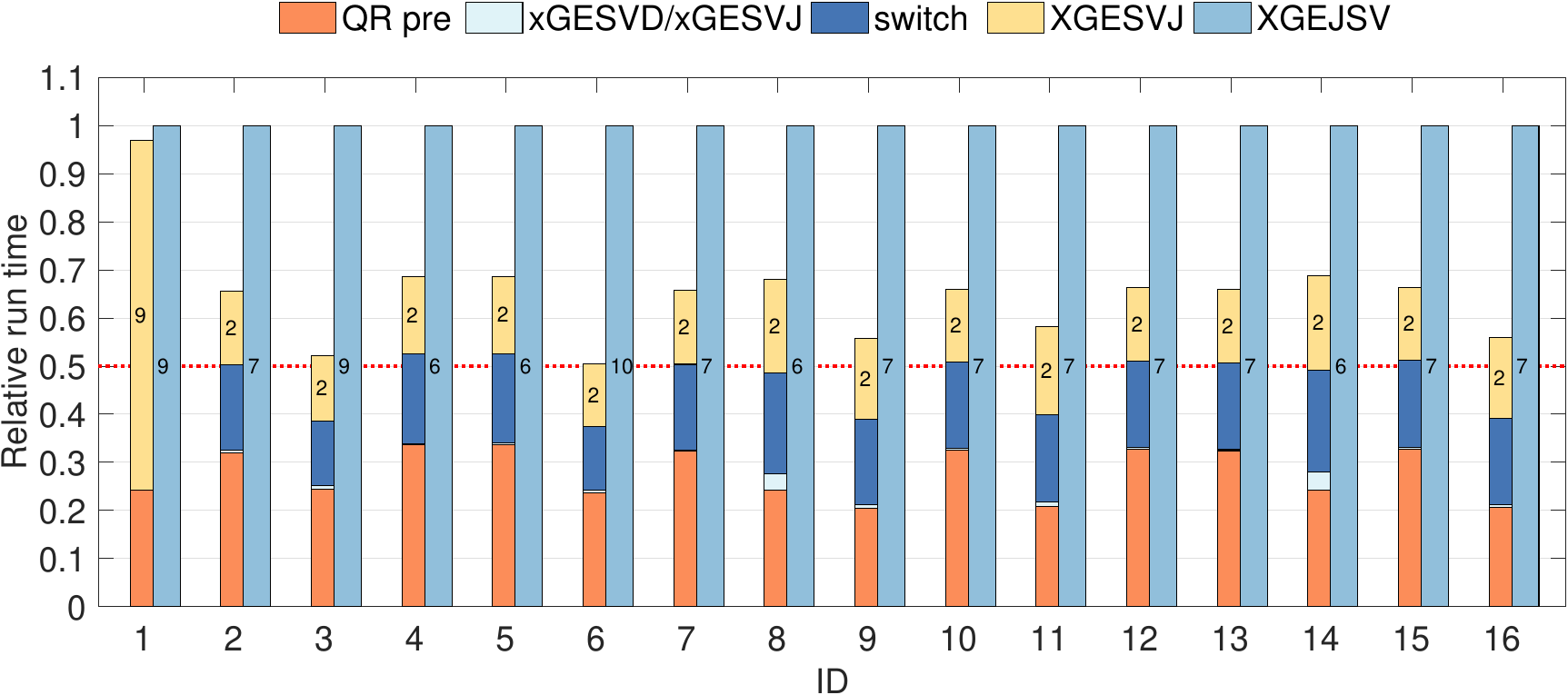}
\caption{Relative run time of Algorithm~\ref{alg:msvj} with
\(\bigl(\kappa(D),\kappa(B)\bigr)=(10^2,10^{12})\) for $1024\times 1024$ real
matrices.
In this case, working and lower precisions are double--double and
IEEE double precision, respectively.
The numbers in the bars denote the numbers of iterations required by the Jacobi algorithm.}
\label{fig:dd_1e21e12_fixedsize}
\end{figure}

We have tested our mixed precision algorithms for a number of matrices with
different sizes.
These test results can be found in the Appendix.

\subsection{Accuracy tests}
In this subsection, we demonstrate that Algorithm~\ref{alg:msvj} computes
singular values to satisfactory relative accuracy.

We test \(16\) types of \(1024\times1024\) matrices as in
Table~\ref{tab:testmat} with
\(\bigl(\kappa(D),\kappa(B)\bigr)=(10^{20},10^2)\) and measure the accuracy of
the computed singular values.
Let \(\hat\lambda_i^{\mathrm{Alg4}}\)'s and
\(\hat\lambda_i^{\mathtt{GEJSV}}\)'s, respectively, be the computed singular
values by Algorithm~\ref{alg:msvj} and \texttt{DGEJSV}.
We measure the relative difference between \(\hat\lambda_i^{\mathrm{Alg4}}\)
and \(\hat\lambda_i^{\mathtt{GEJSV}}\) by
\[
\biggl\lvert\frac{\hat\lambda_i^{\mathrm{Alg4}}
-\hat\lambda_i^{\mathtt{GEJSV}}}
{\hat\lambda_i^{\mathtt{GEJSV}}}\biggr\rvert.
\]
The maximum relative difference observed in our experiments is
\(4.79\times10^{-14}\).

We also measure the backward error
\[
\max_i{\biggl\lbrace\frac{\lVert (A-U\Sigma V\herm)(:, i)\rVert_2}{\lVert A(:, i)\rVert_2}\biggr\rbrace}
\]
and the orthogonality of \(U\) and \(V\) by
\[
\lVert U\herm U-I\rVert_\fro,\quad \lVert V\herm V-I\rVert_\fro.
\]
the largest values of these quantities are \(3.21\times 10^{-14}\),
\(5.85\times 10^{-12}\), and \(9.07\times 10^{-13}\), respectively.
Therefore, Algorithm~\ref{alg:msvj}, as a mixed precision algorithm, achieves
the same level of relative accuracy compared to \texttt{DGEJSV}.

\subsection{Tests of different switching precisions methods}
\label{subsec:experiments-pre}
In this subsection, we test several methods for switching precisions using
test matrices with \(\bigl(\kappa(D),\kappa(B)\bigr)=(10^2,10^{12})\).
We compare three methods as follows.
\begin{enumerate}
\item Compute \(X=U\Sigma V\herm\) in single precision by the Jacobi SVD
algorithm, and then obtain \(Y\gets XQ\), where \(Q\) is \(Q\)-factor in the
QR factorization of \(X\herm U\).
\item Compute \(X=U\Sigma V\herm\) in single precision by the Jacobi SVD
algorithm, and then obtain \(Y\gets XQ\), where \(Q\) is \(Q\)-factor in the
QR factorization of \(V\).
\item Compute \(X=U\Sigma V\herm\) in single precision by the divide and
conquer algorithm, and then obtain \(Y\gets XQ\), where \(Q\) is \(Q\)-factor
in the QR factorization of \(X\herm U\).
\end{enumerate}

In Figure~\ref{fig:pre}, we see that the first method is the most accurate one.
The first method is better than the second one, which is consistent with the
prediction made by Theorem~\ref{thm:effciency_switch}.
The divide and conquer algorithm is not a good choice here since the columns
of \(Y\) are far from orthogonal.

\begin{figure}[!tb]
\centering
\includegraphics[width=\textwidth]{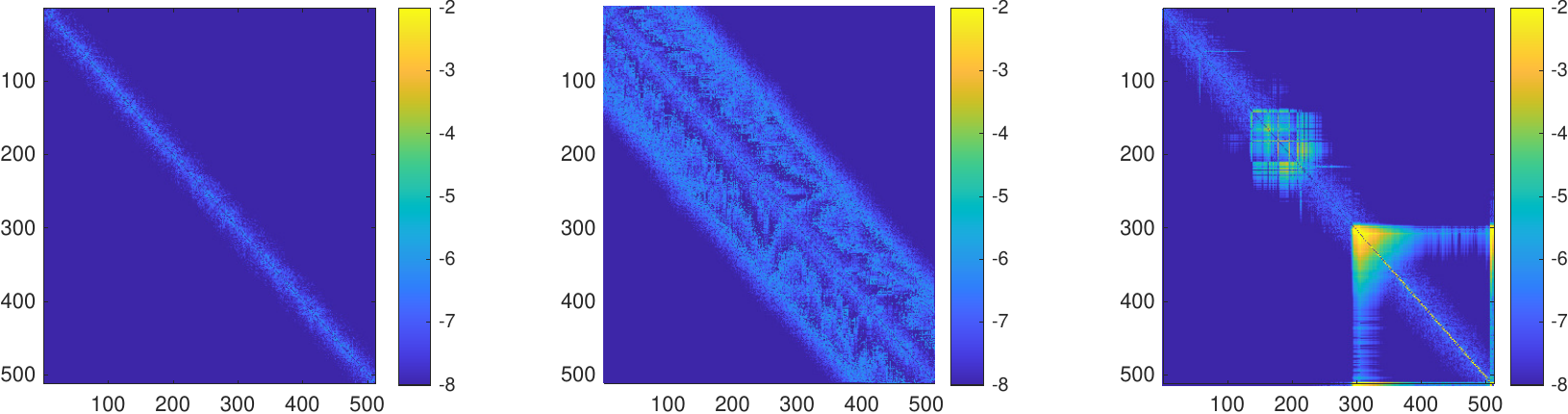}
\caption{Tests of different switching precisions methods: the
componentwise plots of
\(\log_{10}\lvert Y_c\herm Y_c-I\rvert\) produced by MATLAB's \texttt{imagesc},
where \(Y=Y_cD_c\), \(D_c\) is diagonal, and \(Y_c\) has unit column norms.
The subfigures from left to right represent the results of methods 1--3,
respectively.}
\label{fig:pre}
\end{figure}

\section{Conclusions}
\label{sec:conclusions}
In this paper we propose a mixed precision one-sided Jacobi SVD algorithm.
Our algorithm uses QR preconditioning and SVD in a lower precision to produce
an initial guess, and then use the one-sided Jacobi SVD algorithm in
working precision to refine the solution.
Special care is taken in converting the solution from lower precision to
working precision so that the transformed solution is reasonably
accurate and enables fast convergence in the refinement process.
On the x86-64 architecture our algorithm achieves about \(2\times\) speedup
compared to \texttt{DGEJSV}\slash{}\texttt{ZGEJSV} in LAPACK for most test
matrices.
Though a large portion of computations is performed in single precision, the
final solution is as accurate as that produced by double precision
solvers \texttt{DGEJSV}\slash{}\texttt{ZGEJSV}.
This is confirmed both theoretically and experimentally.

Our numerical experiments were performed on the x86-64 architecture, where
single precision arithmetic is usually twice faster compared to double
precision arithmetic.
The situation becomes different on modern GPUs---the performance gap between
lower and higher precision arithmetic is often far larger than \(2\times\),
and more floating-point formats (\texttt{fp64}, \texttt{fp32}, \texttt{fp16},
etc.) are available.
We expect that our algorithm will likely achieve higher speedup on GPUs.
High performance implementations of the mixed precision Jacobi SVD algorithm
on GPUs is planned as our future work.

\ifarxiv
Finally, we remark that it is possible to develop a mixed precision Jacobi
algorithm that computes the spectral decomposition of a Hermitian matrix.
A preliminary discussion can be found in Appendix~\ref{sec:eigen}.
\fi

\section*{Acknowledgments}
The authors thank Zhaojun Bai, Zlatko Drma\v{c}, Daniel Kressner, and Shengguo
Li for helpful discussions.
They are also grateful to the anonymous reviewers for their careful reading and constructive comments.

W.~Gao is partially supported by the National Key R\&D Program of China under
Grant Number 2021YFA1003305.
M.~Shao is partially supported by the National Key R\&D Program of China under
Grant Number 2023YFB3001603.

\ifarxiv

\bibliographystyle{siamplain}
\bibliography{mybib}
\begin{appendix}
\section{Supplementary experiments for SVD problem}
\label{sec:supplementary}
In the following we provide more numerical experiments of the mixed precision
Jacobi algorithm.
In the figures of this section, the \(x\)-axis represents the size of test
matrices and the \(y\)-axis represents the relative run time of
Algorithm~\ref{alg:msvj} comparing with \texttt{XGEJSV} in LAPACK.

Figures~\ref{fig:double_1e21e12_square} and~\ref{fig:double_1e201e2_square}
show the relative run time for different sizes \(\{2048, 4096, 6144, 8192\}\)
real square matrices with \((\kappa(D), \kappa(B)) = (10^2, 10^{12})\) and
\((\kappa(D), \kappa(B)) = (10^{20}, 10^2)\), respectively.
Similarly, Figures~\ref{fig:zplx_1e21e12_square}
and~\ref{fig:zplx_1e201e2_square} show the relative run time for different
sizes \(\{2048, 4096, 6144, 8192\}\) complex square matrices with
\((\kappa(D), \kappa(B)) = (10^2, 10^{12})\) and
\((\kappa(D), \kappa(B))) = (10^{20}, 10^2)\), respectively.
Figure~\ref{fig:double_1e21e12_rectangle} shows the relative run time for
different sizes \(\{6144\times2048, 12288\times4096, 18432\times6144\}\)
real rectangular matrices with \((\kappa(D), \kappa(B)) = (10^2, 10^{12})\).

\begin{figure}[!ptb]
\centering
\includegraphics[width=\textwidth]{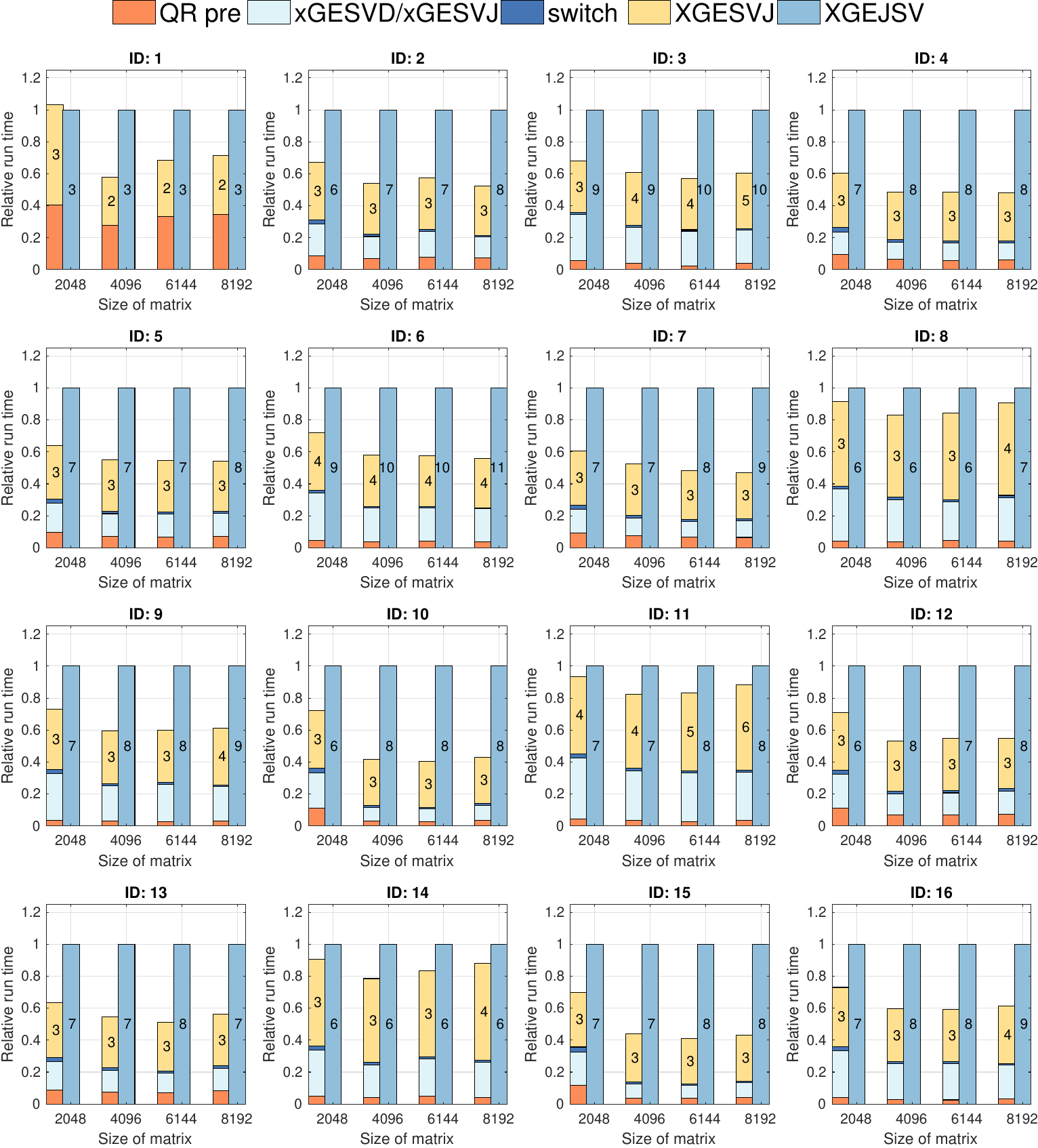}
\caption{Relative run time of Algorithm~\ref{alg:msvj} with
\((\kappa(D), \kappa(B)) = (10^2, 10^{12})\) for real square matrices of
different sizes \(\{2048, 4096, 6144, 8192\}\).
The numbers in the bars denote the numbers of iterations required by the Jacobi algorithm.}
\label{fig:double_1e21e12_square}
\end{figure}

\begin{figure}[!ptb]
\centering
\includegraphics[width=\textwidth]{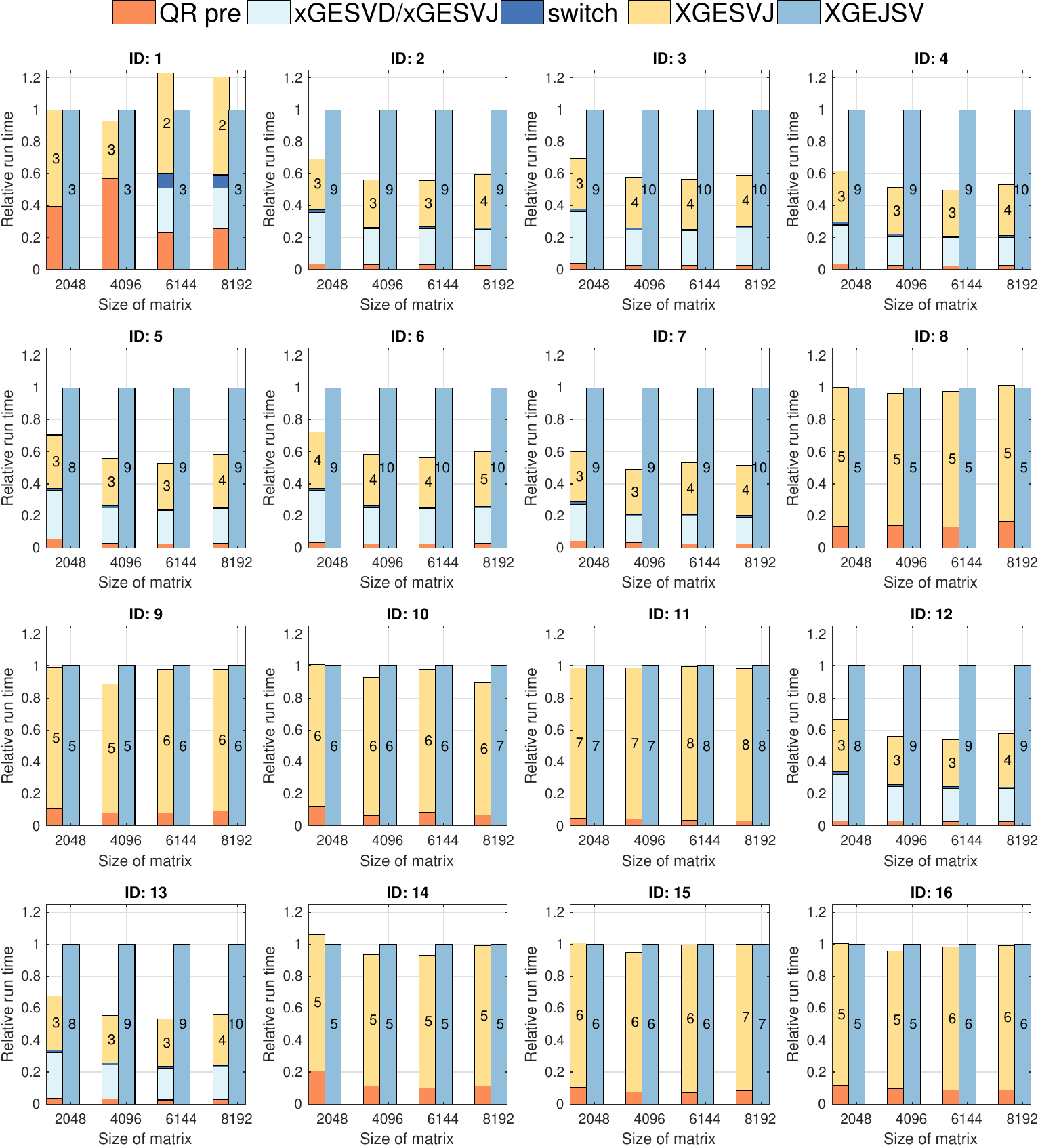}
\caption{Relative run time of Algorithm~\ref{alg:msvj}
with \((\kappa(D), \kappa(B)) = (10^{20}, 10^2)\) for real square matrices of
different sizes \(\{2048, 4096, 6144, 8192\}\).
The numbers in the bars denote the numbers of iterations required by the Jacobi algorithm.}
\label{fig:double_1e201e2_square}
\end{figure}

\begin{figure}[!ptb]
\centering
\includegraphics[width=\textwidth]{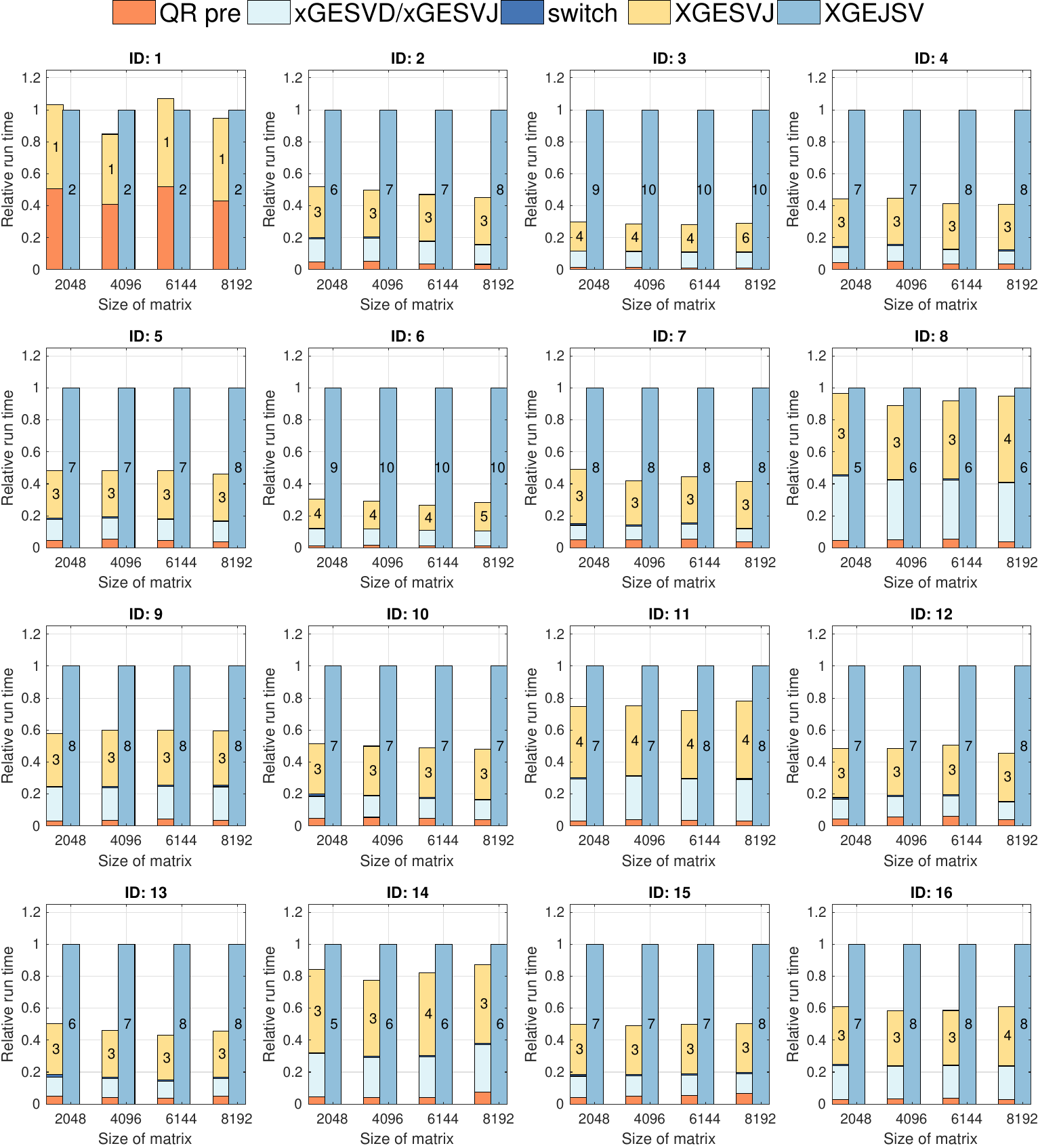}
\caption{Relative run time of Algorithm~\ref{alg:msvj}
with \((\kappa(D), \kappa(B)) = (10^2, 10^{12})\) for complex square matrices
of different sizes \(\{2048, 4096, 6144, 8192\}\).
The numbers in the bars denote the numbers of iterations required by the Jacobi algorithm.}
\label{fig:zplx_1e21e12_square}
\end{figure}

\begin{figure}[!ptb]
\centering
\includegraphics[width=\textwidth]{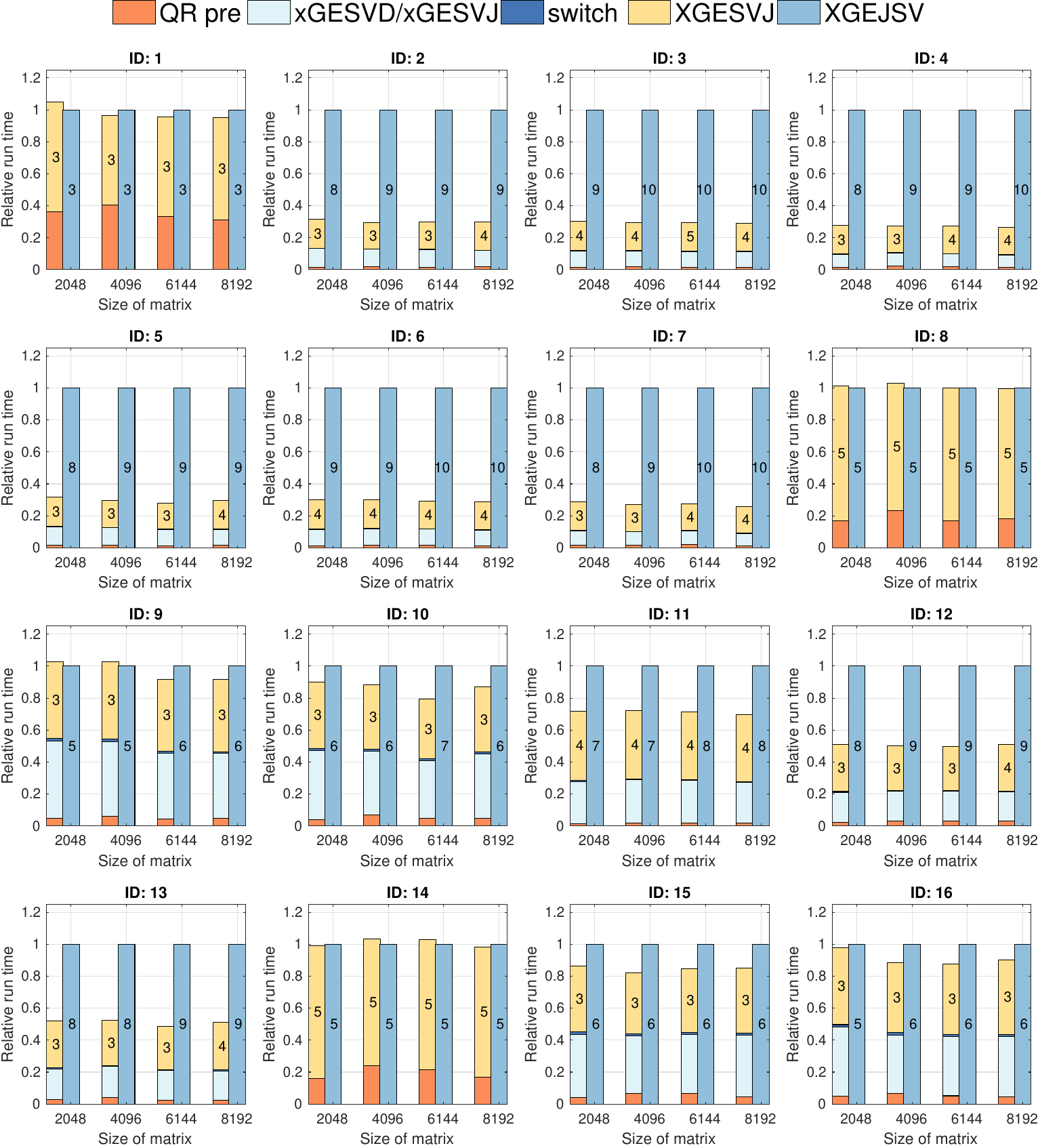}
\caption{Relative run time of Algorithm~\ref{alg:msvj} with
\((\kappa(D), \kappa(B)) = (10^{20}, 10^2)\) for complex square matrices of
different sizes \(\{2048, 4096, 6144, 8192\}\).
The numbers in the bars denote the numbers of iterations required by the Jacobi algorithm.}
\label{fig:zplx_1e201e2_square}
\end{figure}

\begin{figure}[!ptb]
\centering
\includegraphics[width=\textwidth]{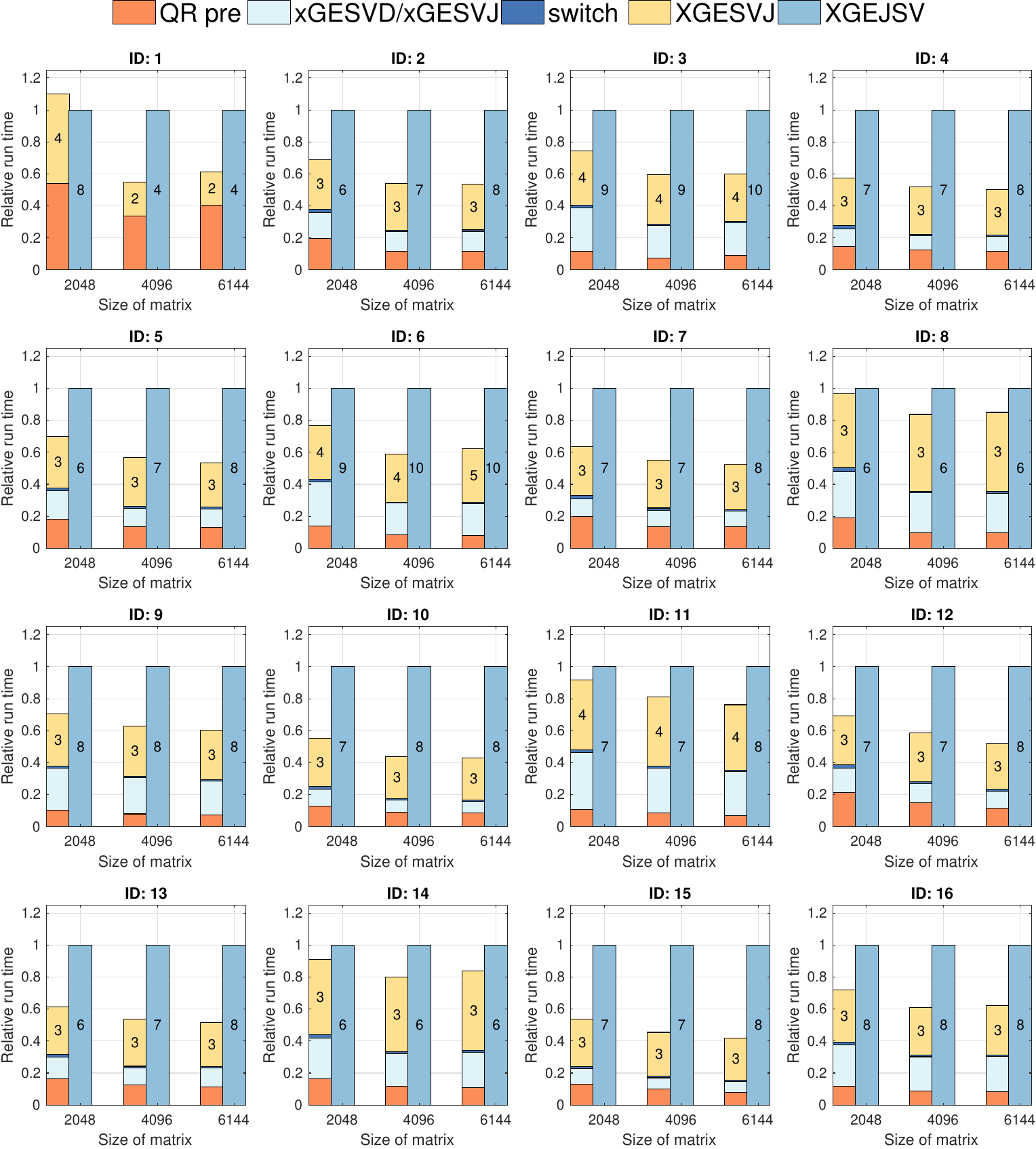}
\caption{Relative run time of Algorithm~\ref{alg:msvj} with
\((\kappa(D), \kappa(B)) = (10^2, 10^{12})\) for real rectangular matrices of
different sizes \(\{6144\times2048, 12288\times4096, 18432\times6144\}\).
The numbers in the bars denote the numbers of iterations required by the Jacobi algorithm.}
\label{fig:double_1e21e12_rectangle}
\end{figure}

\section{Jacobi SVD method for symmetric eigenvalue problem}
\label{sec:eigen}
We consider the symmetric eigenvalue problem
\[
AX = X\Lambda,
\]
where \(A\) is a given Hermitian matrix and $\Lambda$ is a diagonal matrix
with eigenvalues \(\lambda_1\), \(\dots\), \(\lambda_n\).
It is possible to use, for instance, the QR algorithm to compute an
approximate spectral decomposition in lower precision, and then use
the Jacobi algorithm to refine the solution in working precision.
This mixed precision approach, which is essentially a Jacobi algorithm, is
preferred when high relative accuracy of the spectrum is of interest.
However, since two-sided Jacobi eigenvalue algorithm has very low efficiency,
we consider applying the one-sided Jacobi SVD algorithm in the refinement
stage.

Suppose that the SVD of \(A\) is \(A=U\Sigma V\herm\).
A simple case is that \(A\) is either positive semidefinite or negative
semidefinite.
In this case \(\Lambda=U\herm AU\) is diagonal, so that the output from the
one-sided Jacobi SVD algorithm can be directly used.
Complication occurs when \(A\) is indefinite.
Let \(\sigma_1>\sigma_2>\dotsb>\sigma_m\) be the distinct singular values of
\(A\) without counting multiplicity.
Then we have
\begin{equation}
\label{eq:blockdiag}
\tilde A=U\herm AU
=\bmat{\tilde A_1 \\ & \tilde A_2 \\ & & \ddots \\ & & & \tilde A_m},
\end{equation}
where the size of each diagonal block \(\tilde A_i\) is the multiplicity of
the singular value \(\sigma_i\), so that both \(\sigma_i\) and \(-\sigma_i\)
can possibly be eigenvalues of \(A\).
We can identify the size of these diagonal blocks by computing the gap of
consecutive singular values---if two singular values are really close, they
are regarded as the same.
Then each diagonal block is further diagonalized through a usual symmetric
eigensolver in working precision.
This step is expected to be modestly cheap unless \(A^2\) has some eigenvalues
with very high multiplicity.
Finally, as a postprocessing step, the eigenvalues of \(A\) are sorted in
ascending order to match the output format of a usual symmetric eigensolver.

\begin{algorithm}[!tb]
\caption{Mixed precision algorithm for symmetric eigenvalue problem}
\label{alg:jacobi-eigen}
\begin{algorithmic}[1]
\REQUIRE An \(n\)-by-\(n\) Hermitian matrix \(A\).
\ENSURE The spectral decomposition of \(A\).

\STATE Compute the spectral decomposition of \(\single(A)\), i.e.,
    \(\single(A)=X_{\s}\Lambda_{\s}X_{\s}\herm\).
\STATE Orthogonalize the columns of \(X_{\s}\) in working precision,
    \(\double(X_{\s})=QR\), where \(R\) is not explicitly formed.
\STATE Apply the one-sided Jacobi SVD algorithm to \(AQ\) to obtain
    \(AQ=U\Sigma V\herm\), where \(V\) is not explicitly formed.
\STATE Identify the block structure~\eqref{eq:blockdiag}.
\STATE Diagonalize each diagonal block \(\tilde A_i\), and accumulate the
corresponding unitary transformation on \(U\).
\STATE Reorder the eigenvalues and eigenvectors of \(A\).
\end{algorithmic}
\end{algorithm}

We implement the mixed precision one-sided Jacobi eigenvalue algorithm
(i.e., Algorithm~\ref{alg:jacobi-eigen}) based on LAPACK subroutines
\texttt{xSYEVD} and \texttt{XGEJSV}.
Then test matrices with parameters \(\texttt{MODE}_\Sigma=4\),
\(\kappa(B)=10^{16}\) and \(\kappa(D)=1\), i.e. \(A=B\), are used to compare
Algorithm~\ref{alg:jacobi-eigen} with the QR algorithm in terms of accuracy.

From Figure~\ref{fig:eig-error}, we can see that
Algorithm~\ref{alg:jacobi-eigen} is more accurate than the QR algorithm,
especially for eigenvalues with small magnitude.
However, we remark that Algorithm~\ref{alg:jacobi-eigen} is slower than any
popular dense symmetric eigensolver.
We observe that, compared to \texttt{DSYEV} in LAPACK,
Algorithm~\ref{alg:jacobi-eigen} is usually much slower
on CPU;
see Figure~\ref{fig:eig-performance}.
Thus, Algorithm~\ref{alg:jacobi-eigen} is of interest only when eigenvalues
need to be computed to high relative accuracy.

\begin{figure}[!tbh]
\centering
\includegraphics[width=\textwidth]{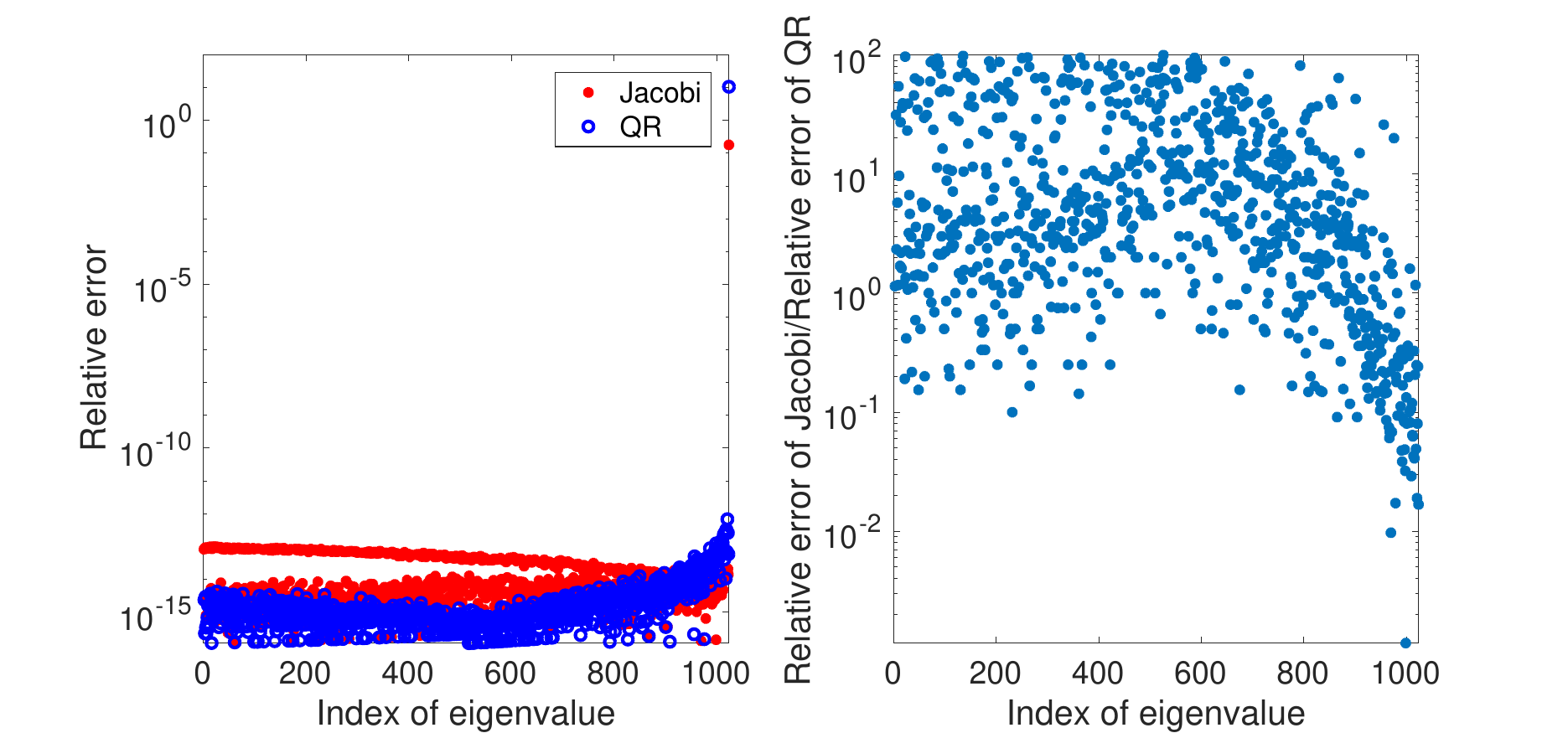}
\caption{Relative error of computed eigenvalues of a \(1024\times1024\) real
symmetric matrix.
Left: Relative errors computed by Jacobi and QR algorithms.
Right: ratio of the relative error of the mixed precision Jacobi algorithm to
that of the QR algorithm for each individual eigenvalue.}
\label{fig:eig-error}
\end{figure}

\begin{figure}[!tbh]
\centering
\includegraphics[width=\textwidth]{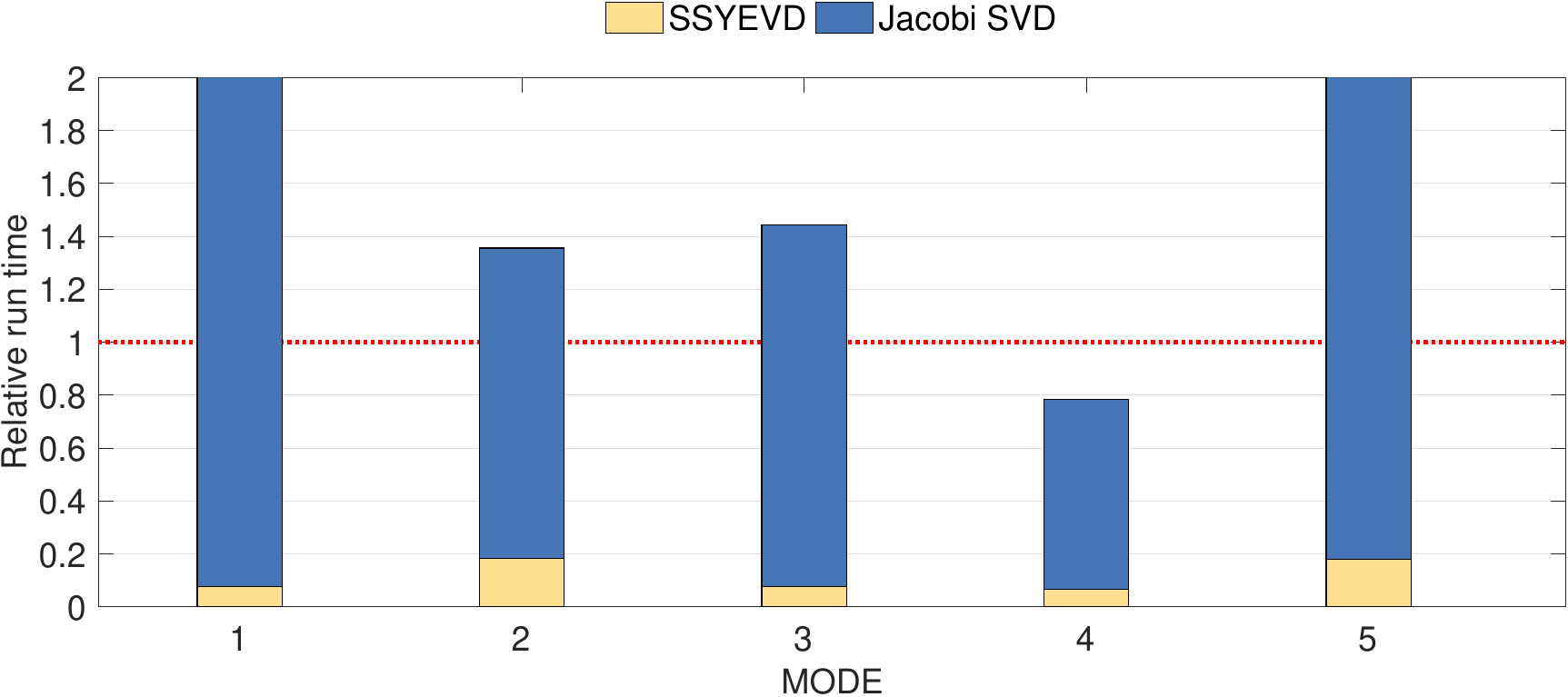}
\caption{Relative run time of Algorithm~\ref{alg:jacobi-eigen} with \(\kappa(A) = 10^{15}\) for real symmetric matrices of the size \(2048\).
Different \texttt{MODE}s represent different eigenvalue distributions; see Table~\ref{tab:XLATM1}.
For \(\texttt{MODE} = 1\) or \(5\), the relative run time exceeds \(2\), rendering it impractical, so the excess was disregarded.}
\label{fig:eig-performance}
\end{figure}

\end{appendix}

\fi

\end{document}